\documentclass{amsart}
\UseRawInputEncoding
\usepackage{amsmath,empheq, amsthm}
\usepackage[utf8]{inputenc}
\usepackage{comment}
\usepackage{amscd}
\usepackage{todonotes}
\usepackage{amssymb}
\usepackage{pstricks}

\usepackage{mathrsfs}
\usepackage{graphicx}
\usepackage{lineno}
\usepackage{subcaption}
\usepackage{todonotes}
\newtheorem{definition}{Definition}
\newtheorem{theorem}{Theorem}

\newtheorem{lemma}{Lemma}

\newtheorem{remark}{Remark}
\newtheorem{example}{Example}

\newcommand{\R}{\mathbb{R}}
\numberwithin{equation}{section}

\title[]{%Smoothnesss of topological conjugacy by a mixed dichotomy and bounded growth approach
Smoothness of linearization by mixing parameters of dichotomy, bounded growth and perturbation}
\thanks{The first author was funded by FONDECYT Regular Grant 1240361. The second author was funded by FONDECYT Postdoctoral Grant 3210132. The third author was funded by FONDECYT Regular Grant 1210733} 
\author[]{\'Alvaro Casta\~neda}
\author[]{Ignacio Huerta}
\author[]{Gonzalo Robledo}

\address{Universidad de Chile, Departamento de Matem\'aticas. Casilla 653, Santiago, Chile}
\address{Departamento de Matem\'atica, Universidad T\'ecnica Federico Santa Mar\'ia, Casilla 110-V, Valpara\'iso, Chile}
\email{castaneda@uchile.cl, ignacio.huertan@usm.cl grobledo@uchile.cl }

\begin{document}

\begin{abstract}
We study the smoothness properties of a global and non\-au\-to\-no\-mous topological conjugacy
between a linear system and a quasilinear perturbation. The linear system exhibits a nonuniform exponential dichotomy with a nontrivial projector and nonuniform bounded growth property. Additionally, the quasilinear perturbation is dominated by an increasing exponential function. Emphasis is placed on employing a set of parameters to describe the conditions of dichotomy, bounded growth and quasilinear perturbations. Finally, we prove that modifying these conditions enables us to achieve a broader smoothness interval.

% We study the smoothness properties
% of a global and nonautonomous version of the Hartman–Grobman Theorem
% when the linear system has a nonuniform exponential dichotomy with nontrivial projector, nonuniform bounded growth and the nonlinear perturbation is unbounded. The use of a set of parameters to describe the hypotheses is emphasized, and also that the modification of these conditions will allow us to obtain a larger smoothness interval.  
\end{abstract}

\maketitle

\section{Introduction}
The study of the smoothness properties of the topological conjugacy is relatively recent in the literature 
of nonautonomous systems and a sequence
of accumulative results have been obtained is the last years. This article can be seen as a
continuation of this trend and its main contribution
is a set of sharper results for a specific case.

Roughly speaking, the topological conjugacy refers to the existence of a time--parametrized family of $\mathbb{R}^{n}$--homeomorphisms mapping solutions of a linear system into solutions of a quasilinear perturbed one and vice versa. In this work, we revisit the problem of the topological conjugacy between the following systems of nonautonomous ordinary differential equations:
\begin{subequations}
  \begin{empheq}[left=\empheqlbrace]{align}
    & \dot{x} = A(t)x, \label{lin1} \\
    & \dot{y} = A(t)y + f(t,y), \label{nolin1}
  \end{empheq}
\end{subequations}
where $A \colon \mathbb{R}_0^+ \to M_n(\mathbb{R})$ 
and $f: \mathbb{R}_{0}^{+} \times \mathbb{R}^n \to \mathbb{R}^n$ are continuous functions such that the existence, uniqueness and unbounded forward continuability of solutions is ensured. Moreover $\mathbb{R}_{0}^{+}:=[0,+\infty)$
and $M_{n}(\mathbb{R})$ is the set of square 
$n\times n$ matrices with real coefficients. A fundamental matrix associated to (\ref{lin1})
will be denoted by $\Phi(t)$ while $\Phi(t,s):=\Phi(t)\Phi^{-1}(s)$ is the corresponding evolution operator. From now on, the symbol $|\cdot|$ will denote a norm in $\mathbb{R}^{n}$ while $\|\cdot\|$ will be its induced matrix norm.

We will consider the following formal definition of topological conjugacy:
\begin{definition}
\label{eqtopnouniforme}
The systems \textnormal{(\ref{lin1})} and \textnormal{(\ref{nolin1})} will be called topologically conjugated if there exists a map $H:\R_0^{+}\times\R^n \to \mathbb{R}^{n}$ with the following properties:
\begin{itemize}
    \item [\textbf{(TC1)}] For each fixed $t\in\R_0^{+}$, the map $\xi\mapsto H_{t}(\xi):=H(t,\xi)$ is an homeomorphism and its inverse is denoted by $\xi\mapsto H_{t}^{-1}(\xi):=H^{-1}(t,\xi)=G(t,\xi):=G_{t}(\xi)$.  
      \item [\textbf{(TC2)}] If $t\mapsto x(t,\tau,\xi)$ is the solution of \textnormal{(\ref{lin1})} passing through $\xi$ at $t=\tau$, then 
    $t\mapsto H(t,x(t,\tau,\xi))$ is solution of \textnormal{(\ref{nolin1})}  passing through $H(\tau,\xi)$ at $t=\tau$, that is
 \begin{equation} \label{conjugacion}
H(t,x(t,\tau,\xi))=y(t,\tau,H(\tau,\xi)) \quad \text{for all}\; t\in\R_0^{+}.
\end{equation}

\noindent    Additionnally, if $t\mapsto y(t,\tau,\xi)$ is the solution of \textnormal{(\ref{nolin1})}  passing through $\xi$ at $t=\tau$, then $t\mapsto G(t,y(t,\tau,\xi))$ is the solution of \textnormal{(\ref{lin1})} passing through $G(\tau,\xi)$ at $t=\tau$, that is
\begin{equation}
\label{conj2}
G(t,y(t,\tau,\xi))=x(t,\tau,G(\tau,\xi))=\Phi(t,\tau)G(\tau,\xi) \quad \textnormal{for all $t\in \mathbb{R}_{0}^{+}$}.
\end{equation}
 %   \item [\textnormal{(ii)}] For any fixed $t\in\R_0^{+}$, the maps $\xi\mapsto H(t,\xi)$ and $\xi\mapsto H^{-1}(t,\xi)=G(t,\xi)$ are continuous.
    \item [\textbf{(TC3)}] If $|\xi|\rightarrow+\infty$, then $|H(t,\xi)|\rightarrow+\infty$ for any fixed $t\geq 0$.
 %   \item [\textnormal{(iv)}] If $x(t)$ is a solution of \textnormal{(\ref{lin1})}, then $H(t,x(t))$ is a solution of \textnormal{(\ref{nolin1})}. Similarly, if $y(t)$ is a solution of \textnormal{(\ref{nolin1})}, then $G(t,y(t))$ is a solution of \textnormal{(\ref{lin1})}.
  
\end{itemize}
\end{definition}

It is important to emphasize that there doesn’t exist a univocal definition of topological conjugacy and the literature shows slight variations of Definition \ref{eqtopnouniforme}, which are mainly focused in properties that are equivalent or more general than \textbf{(TC3)}. For example, in \cite{PalmerEqTop} it is also stated that if $|\xi|\to 0$ then $|H(t,\xi)|\to 0$ for any fixed $t\geq 0$. Moreover, in \cite{Palmer} it is stated that 
$|\xi-H(t,\xi)|$ is bounded for any $t$ and $\xi$, which implies \textbf{(TC3)}.

The problem of determining conditions on (\ref{lin1})--(\ref{nolin1}) ensuring the existence of a topological conjugacy 
has been addressed by several ways. For example, the \textit{crossing time approach} and the \textit{Green function approach}, both inspired in the classical Theorem of Hartman--Grobman \cite{Grobman, Hartman1}, which constructs a local topological conjugacy between a nonlinear system and its linearized counterpart in a neighborhood of an hyperbolic equilibrium point.

The crossing time approach --roughly speaking-- assumes that the origin is a
globally uniformly asymptotically stable equilibrium of (\ref{lin1})--(\ref{nolin1}). This implies that any solution
of these systems, either forwardly of backwardly have a unique intersection with the unit sphere at a time named \textit{crossing time}, which are the basis for the construction of the homeomorphisms. We refer the reader to \cite{Lin,PalmerEqTop} for some examples
and \cite{Huerta, WuXia} for some generalizations.

The Green function approach backs to the seminal work of K.J. Palmer \cite{Palmer}, which assumes that the linear system
(\ref{lin1}) has a property of exponential dichotomy (a formal definition will be given in the next section) which emulates some qualitative features of the hyperbolicity to the nonautonomous framework. In addition, it is assumed that the quasilinear perturbation of (\ref{nolin1}) verifies some Lipschitzness and boundedness properties. A noticeable feature of this approach is that provides an explicit construction of the homeomorphisms $\xi \mapsto G(t,\xi)$ and $\xi \mapsto H(t,\xi)$. 

This work follows the approach based in the Green function and  --in order to contextualize its novelty-- 
we will describe the progress made in the study of these homeomorphisms.
\newpage

\subsection{The Green function approach} 

\subsubsection{State of art of continuity results}
This approach assumes that the linear system (\ref{lin1}) has a property of dichotomy on the interval $J\subseteq \mathbb{R}$ while the quasilinear perturbation
$x\mapsto f(t,x)$ verify the following properties for any $t\in J$ and $x,\tilde{x}\in \mathbb{R}^{n}$:
\begin{equation*}
%\label{nolinealidad}
|f(t,x)-f(t,\tilde{x})|\leq L(t)|x-\tilde{x}| \,\, \textnormal{and} \quad |f(t,x)|\leq B(t),   
\end{equation*}
where $L\colon J\to (0,+\infty)$ and $B\colon J\to (0,+\infty)$ satisfy the additional condition for any fixed $t\in J$:
\begin{equation}
\label{GVNL}
\int_{J}\|\mathcal{G}(t,s)\|L(s)\,ds =q<1  \,\, \textnormal{and} \,\,  \int_{J}\|\mathcal{G}(t,s)\|B(s)\,ds =p< +\infty, 
\end{equation}
and $\mathcal{G}(t,s)$ is the Green function associated to the corresponding dichotomy on $J$, which will be 
formally stated in the next section.

%The Hartman-Grobman Theorem \cite{Grobman, Hartman1} is a cornerstone in the analysis of nonlinear autonomous dynamical systems, providing an important link between the local stability of a nonlinear system and the behavior of its linearized counterpart near of an equilibrium point.

% The Hartman-Grobman Theorem \cite{Grobman, Hartman1} is fundamental in the study of nonlinear autonomous dynamical systems because it establishes a crucial connection between the local stability of a linear system and its linearized behavior around an equilibrium point. 
The seminal work of Palmer constructs an explicit homeomorphism by considering an exponential dichotomy on $J=\mathbb{R}$
and supposing that $L(\cdot)$ and $B(\cdot)$ are constant functions. The Palmer's result is inspired in a previous work of global linearization obtained by C. Pugh \cite{Pugh} 
in the autonomous framework, where the linear part is hyperbolic. We stress that A. Reinfelds constructed a relatively
similar pair of homeomorphisms in \cite{Reinfelds}. Additionally, J. Shi and K. Xiong \cite{Shi} noted that in many cases, the condition \textbf{(TC3)} for topological conjugacy can be strengthened, ensuring that the homeomorphism $\xi\mapsto H(t,\xi)$ exhibits uniform continuity and/or H\"older continuity with respect to the variable $\xi$. 

\medskip
The explicit construction of the Palmer's homeomorphism has become a method which has been 
progresively improved to carry out linearization results
for systems (\ref{lin1})--(\ref{nolin1}) whose linear part have dichotomies more general than the exponential one, which impose 
stronger conditions for the maps $t\mapsto L(t)$ and  $t\mapsto B(t)$. We refer the reader to \cite{BentoSilva, ChenXia, Huerta, Jiang, Lu, RS, ZFZ} for a representative set of these generalizations.

\subsubsection{State of art of smoothness results}

Despite the impressive amount of research devoted to the linearization via the Green function approach, the study 
of its smoothness properties is relatively recent and less extensive.

A first group of results assumed, analogously to the Hartman's approach \cite{Hartman1}, a contractive context, namely, the linear system (\ref{lin1}) is asymptotically stable in some sense: in \cite{CRDif} it was assumed that (\ref{lin1}) is uniformly exponentially stable, namely, the linear system exhibits an exponential dichotomy over $\mathbb{R}$ with the identity projector, achieving a $C^{2}$--global linearization under more restrictive conditions on the quasilinear perturbation. This result was refined in \cite{CRM}, considering an exponential dichotomy in $\mathbb{R}_0^{+}$ which simplified the technical assumptions regarding nonlinearities. Additionally, it was shown that the homeomorphism is $C^{r}$, with $r\geq1$, in a straightforward manner. The results from \cite{CRM} were generalized in \cite{CRM2} where it was assumed that (\ref{lin1}) is nonuniformly asymptotically stable. 

Recently, by considering the nontrivial projectors case, in \cite{Jara} N. Jara extended the above results obtaining a $C^{2}$ class of the homeomorphism by assuming a more general definition of nonuniform exponential dichotomy. 

For another works devoted to the smoothness of the topological conjugacy beyond the Green function approach, we refer the reader to \cite{CDS, Dragicevic, David}.

\subsection{Novelty of this work}  Similarly to \cite{Jara}, this work addresses the smoothness of a topological conjugacy by considering a dichotomy with nontrivial projector. Nevertheless, we stress that the differences are the following: i) we restrict our interest to the nonuniform exponential dichotomy on $\mathbb{R}_{0}^{+}$ for the linear system (\ref{lin1}), ii) we add a supplementary assumption of nonuniform bounded growth for (\ref{lin1}) and iii) we work with functions $t\mapsto L(t)$ and $t\mapsto B(t)$ whose properties are less restrictive than those assumed in \cite{Jara}. In particular, we explore if we consider classes of unbounded functions $t\mapsto B(t)$, to what extent smoothness results can be obtained.

\subsection{Structure}
In this paper we seek to obtain a new version of a linearization theorem in the nonautonomous context, where we will also study the smoothness of the homeomorphism that defines the topological conjugacy. 

The section 2 describes the technical aspects involved in this work; definitions of nonuniform exponential dichotomy, nonuniform bounded growth and the statement of the first main theorem, which establishes the existence of topological conjugacy between \eqref{lin1} and \eqref{nolin1}.
The details of all the steps necessary to achieve the proof of the first main theorem can be found in section 3.

Latterly, in section 4, the second main theorem of this work is presented, where the conditions that allow us to affirm that the homeomorphism that defines the topological conjugation is, indeed, a diffeomorphism are sought.
The article ends with section 5, focused on remarks about the two main results and a comparison with those in \cite{Jara} that have the same implications.

% \subsection{Notations}
% Throughout this paper, $|\cdot|$ will denote a vector norm whose induced matrix norm is given by $\|\cdot\|$. The set $[0,+\infty)$ is denoted by $\mathbb{R}_{0}^{+}$ and the set of square 
% $n\times n$ matrices with real coefficients is denoted by $\mathcal{M}_{n}(\mathbb{R})$, while $I$ is the identity matrix.

\section{Preliminaries.}
%Roughly speaking, the topological conjugacy refers to the existence of a time--parametrized family of $\mathbb{R}^{n}$--homeomorphisms mapping solutions of a linear system into solutions of a quasilinear perturbed one and vice versa. In this work, we revisit the problem of the topological conjugacy between the following systems of nonautonomous ordinary differential equations:
%\begin{subequations}
%  \begin{empheq}[left=\empheqlbrace]{align}
%    & \dot{x} = A(t)x, \label{lin1} \\
%    & \dot{y} = A(t)y + f(t,y), \label{nolin1}
%  \end{empheq}
%\end{subequations}
%where $A \colon \mathbb{R}_0^+ \to M_n(\mathbb{R})$ 
%and $f: \mathbb{R}_{0}^{+} \times \mathbb{R}^n \to \mathbb{R}^n$ are continuous functions such that the existence, uniqueness and %unbounded forward continuability of solutions is ensured. A fundamental matrix associated to (\ref{lin1})
%will be denoted by $\Phi(t)$ while $\Phi(t,s)=\Phi(t)\Phi^{-1}(s)$ is the corresponding evolution operator.

In the literature devoted to the topological conjugacy problem, it is common to assume that the linear system (\ref{lin1}) has the properties of dichotomy and bounded growth in some sense. On the other hand, it is also assumed that the nonlinear term of (\ref{nolin1}) has boundedness and Lipschitzness type properties.

\subsection{Nonuniform exponential dichotomy and nonuniform bounded growth}
The following two definitions will describe the assumptions that we will impose on system (\ref{lin1}) in this work.

\begin{definition}\textnormal{(\cite{Chu}, \cite{Zhang})}
\label{NUEDD}
The system \textnormal{(\ref{lin1})} has a nonuniform exponential dichotomy on $\R_0^+$ if there exist an invariant projector $P(\cdot)$, constants $K\geq 1$, $\alpha>0$, $\mu\geq 0$ with $\mu<\alpha$ such that
\begin{equation}
\left\{\begin{array}{rcl}
\label{NUED}
P(t)\Phi(t,s)&=&\Phi(t,s)P(s), \quad t, s \in \R_0^+, \\
\left \| \Phi(t,s)P(s) \right \|&\leq& Ke^{-\alpha(t-s)+\mu s},\quad t\geq s, \quad t, s \in \R_0^+,                      \\
\left \| \Phi(t,s)(I-P(s)) \right \|&\leq& Ke^{\alpha(t-s)+\mu s},\quad t\leq s, \quad t, s \in \R_0^+,
\end{array}\right.
\end{equation}
where $K$, $\alpha$ and $\mu$ are called the constants of the nonuniform exponential dichotomy.
\end{definition}

\begin{remark}The property of nonuniform exponential dichotomy deserves a few comments:
\begin{enumerate}
\item[i)] A consequence of the first equation of \eqref{NUED} is that $\dim\ker P(t)=\dim\ker P(s)$ for all $t,s\in \mathbb{R}_{0}^{+};$ this motives that,  in the literature, the projector $P(\cdot)$ is known as invariant projector. 
\item[ii)] If $\mu=0$, we recover the classical exponential dichotomy, also called uniform exponential dichotomy \cite{Palmer,PalmerEqTop}. In this framework, the function $s\mapsto e^{\mu s}$ is also known as the nonuniform part.
\item[iii)] We emphasize that in the literature there are some definitions of nonuniform exponential dichotomy that come without the requirement that $\mu<\alpha$, but there are certain aspects of this theory that require this condition (see \cite{JG} for more details).
\end{enumerate}
\end{remark}

\begin{remark}
We emphasize that in \cite[p.540]{Chu} it is stated that the nonuniform exponential dichotomy is admitted by any linear system with nonzero Lyapunov exponents.
\end{remark}

\begin{remark}
The nonuniform exponential dichotomy on $\mathbb{R}_{0}^{+}$ with a non trivial projector $P(\cdot)$ 
implies that any non zero solution $t\mapsto x(t,t_{0},\xi)=\Phi(t,t_{0})\xi$ can be splitted 
as $x(t,t_{0},\xi)=x^{+}(t,t_{0},\xi)+x^{-}(t,t_{0},\xi)$, where
$$
t\mapsto x^{+}(t,t_{0},\xi)=\Phi(t,t_{0})P(t_{0})\xi \quad
\textnormal{and}  \quad t\mapsto x^{-}(t,t_{0},\xi)=\Phi(t,t_{0})[I-P(t_{0})]\xi
$$  
such that $t\mapsto x^{+}(t,t_{0},\xi)$ is a contraction, which converges nonuniformly exponentially to the origin when $t\to +\infty$ while $t\mapsto x^{-}(t,t_{0},\xi)$ is an expansion, which diverges nonuniformly exponentially. The convergence as well as the divergence are dependent of $t-t_{0}$, namely, the elapsed time between $t$ and $t_{0}$ but are also dependent of the initial time $t_{0}$. In this context, the behavior is not uniform with respect to $t_{0}$. Last but not least,
this ``dichotomic" and ``nonuniform" a\-symp\-to\-tic exponential behavior justifies the name of nonuniform exponential dichotomy.
\end{remark}

Additionally, the Definition \ref{NUEDD} induces the construction of the Green's operator associated to the dichotomy property as the function $\mathcal{G}:\R_{0}^{+}\times\R_{0}^{+}\to M_{n}(\R^{n})$, described by:
\begin{equation}
\label{Green}
    \mathcal{G}(t,s)=\left \{ \begin{array}{rcl}
    \Phi(t,s)P(s),\quad &\textnormal{if}&\; t\geq s\geq 0,\\
    -\Phi(t,s)Q(s),\quad &\textnormal{if}&\; s\geq t\geq 0,
    \end{array}
    \right.
\end{equation}
where $Q(\cdot)=I-P(\cdot)$ is the complementary projector.

\begin{definition} The linear system \textnormal{(\ref{lin1})}  has a nonuniform bounded growth on $\mathbb{R}_{0}^{+}$ if there exist constants $K_{0}\geq1$, $a>0$, and $\varepsilon\geq 0$ such that 
\begin{equation}
\label{NUBG}
\|\Phi(t,s)\|\leq K_0e^{a|t-s|+\varepsilon s},\quad \textnormal{for any}\; t,s\in\mathbb{R}_0^{+},
\end{equation}
where $K_{0}$, $a$ and $\varepsilon$ are called the bounded growth constants.
\end{definition}

% \textcolor{red}{
\begin{remark}
%label{NUSED}
In the works \cite{Dragicevic} and \cite{Lu} the concept of nonuniform strong exponential dichotomy for the system \eqref{lin1} is introduced, namely, the system \eqref{lin1} admits a nonuniform exponential dichotomy and a nonuniform bounded growth, where $\alpha\leq a$.   
\end{remark}
% }

When (\ref{lin1}) has a uniform strong exponential dichotomy, namely when $\mu=\varepsilon=0$, it is well known that 
$\alpha\leq a$ and we refer to Shi and Xiong \cite[p.823]{Shi}. This property has been generalized to the nonuniform framework in \cite{CHR1} and extended in \cite{JG}, where the relation between the properties of (nonuniform) exponencial dichotomy and bounded growth is described. Nevertheless, we will revisit it in order to make this article self contained:
\begin{lemma}
If the linear system \eqref{lin1} has the properties of nonuniform exponential dichotomy and nonuniform bounded growth on $\mathbb{R}_{0}^{+}$ with constants $(K,\alpha,\mu)$ and $(K_{0},a,\varepsilon)$ respectively, then it follows that 
\begin{equation}
\label{constraints-IH}
a+\max\{\mu,\varepsilon\}\geq \alpha.
\end{equation}
\end{lemma}

\begin{proof} The proof will be made by contradiction: we will assume that (\ref{constraints-IH}) is not verified and we have that $\alpha >   \max\{\mu,\varepsilon\} + a   $, which implies the inequalities
\begin{equation}
\label{CIH1}
\alpha >  \mu + a   \quad \textnormal{and} \quad 
\alpha >   \varepsilon + a.  
\end{equation}

Firstly, let us consider the case $t\geq s$: by using properties of matrix norms combined with the invariance property, we can deduce that
\begin{displaymath}
\begin{array}{rcl}
\|I\| & = & \|\Phi(t,s)\,\Phi(s,t)\|\\\\
& = & \|\Phi(t,s)[P(s)+Q(s)]\Phi(s,t)\|\\\\
&\leq &  \|\Phi(t,s)P(s)\|\, \|\Phi(s,t)\|+\|\Phi(t,s)\|\,\|\Phi(s,t)Q(t)\|,
\end{array}
\end{displaymath}
and by using the dichotomy and bounded growth properties, we ensure the following estimate:
\begin{displaymath}
\begin{array}{rcl} 
\|I\| &\leq& \displaystyle KK_{0}e^{-\alpha(t-s)+\mu s}e^{a(t-s)+\varepsilon t}+KK_{0}e^{-\alpha(t-s)+\mu t}e^{a(t-s)+\varepsilon s} \\\\
&= & \displaystyle KK_{0}\left[e^{(-\alpha+\varepsilon+a)t}e^{(\mu+\alpha-a)s}+e^{(-\alpha+\mu+a)t}e^{(\alpha-a+\varepsilon)s}\right].
\end{array}
\end{displaymath}

Notice that (\ref{CIH1}) is equivalent to $-\alpha+\mu+a <0$ and $-\alpha+\varepsilon+a<0$. Then letting $t\to +\infty$ leads to
$\|I\|\leq 0$, obtaining a contradiction.

\end{proof}

\begin{remark}
If the linear system \eqref{lin1} has a property of nonuniform bounded growth
on $[0,+\infty)$, an open problem is to determine what additional conditions are necessary to ensure a nonuniform exponential dichotomy
on $[0,+\infty)$. This problem has been solved in the uniform framework, where it was proved that if 
the uniform bounded growth is verified and there exist
$\theta \in (0,1)$ and $T_{\theta}>0$ such that any solution $t\mapsto x(t)$ of \eqref{lin1} verifies
$$
|x(t)|\leq \theta \sup\limits_{|u-t|\leq T_{\theta}}|x(u)| \quad \textnormal{for any $t\geq T_{\theta}$},
$$
then there are exponential dichotomy on $[0,+\infty)$. We refer to \cite[pp. 14-15]{Coppel} and \cite{Palmer2006} for details.
\end{remark}

In parallel, it is usual
in the study of topological linearization to assume that the nonlinear pertubation in (\ref{nolin1}) has certain Lipschitz and boundedness type properties. In particular, throughtout this article, it will be assumed that

% \begin{subequations}
%   \begin{empheq}[left=\empheqlbrace]{align}
%     & \dot{y} =[C(t)+B(t)]y \label{lin2} \\
%     & \dot{y} = [C(t)+B(t)]y + g(t,y) \label{nolin2}
%   \end{empheq}
% \end{subequations}
% where $B, C \colon \mathbb{R}_0^+ \to M_n(\mathbb{R})$ and $g: \mathbb{R}_{0}^{+} \times \mathbb{R}^n \to \mathbb{R}^n$ is continuous on $(t,x).$ Additionally, the following properties are verified:

\begin{itemize}
    \item [\textbf{(P1)}]
    There exist $L_{f}>0$ and $\theta\geq 0$ such that
    $$
    |f(t,y_1)-f(t,y_2)|\leq L_{f}e^{-\theta t}|y_1-y_2| \quad \textnormal{for any $t\geq 0$ and $y_{1},y_{2}\in \mathbb{R}^{n}$}.
    $$
    \item [\textbf{(P2)}] 
     There exist $M>0$ and $\delta\geq 0$ such that
    $$
    |f(t,y)|\leq Me^{\delta t} \quad \textnormal{for any $t\geq 0$ and $y\in \mathbb{R}^{n}$}.
    $$  
\end{itemize}

\begin{example} An example of function $f$ satisfying \textnormal{\textbf{(P1)}} and \textnormal{\textbf{(P2)}} is given by 
$$
f(t,x)=e^{-\theta t}f_{0}(t,x)+M_{0}e^{\delta t},
$$
where $M_{0}>0$ and $f_{0}\colon \mathbb{R}_{0}^{+}\times \mathbb{R}^{n}\to \mathbb{R}^{n}$ is a bounded function with respect to $(t,x)$ and a Lipschitzian function with respect to $x$ with Lipschitz constant $L_{f_{0}}>0$.
\end{example} 
% \todo{Falta decir que $f_0$ es acotada con respecto a la variable $x$, que su cota que depende de $t$, es del tipo exponencial, y falta especificar quién este $I$, debiese ser un vector constante de n coordenadas, pero falta especificarlo.}

Despite the existence of a wide range of results devoted to the continuous dependence of the solutions of  (\ref{nolin1}) with respect to the initial conditions, a careful revision of the literature shows that the results are scarce when assuming that (\ref{lin1}) has the property of nonuniform bounded growth. The next result considers this assumption and is tailored for \textbf{(P1)--(P2)}.
\begin{lemma}
\label{continuidadCIy}
Assume that system \eqref{lin1} admits nonuniform bounded growth, properties \textnormal{\textbf{(P1)}} and \textnormal{\textbf{(P2)}} are satisfied and $\varepsilon<\theta$, then for any $s,t\geq0$ and  $\eta,\tilde{\eta}\in\R^{n}$, we have that the solutions of \eqref{nolin1} passing respectively trough $\eta$ and $\tilde{\eta}$ at $s=t$ verify
\begin{equation}
\label{ydiferencia1}
    |y(s,t,\eta)-y(s,t,\tilde{\eta})|\leq 
    K_0 e^{a|t-s|+\varepsilon t}|\eta-\tilde{\eta}|e^{\frac{K_0L_{f}}{\theta-\varepsilon}} 
\end{equation}
% or
% \begin{equation}
% \label{ydiferencia2  }
%     |y(s,t,\eta)-y(s,t,\tilde{\eta})|\leq 
%     \textnormal{\textcolor{red}{work in progress}} \quad \textnormal{if}\quad  \varepsilon \geq \theta 
% \end{equation}
\end{lemma}
% \todo{Es necesario el caso cuando $\varepsilon\geq\theta$??}

\begin{proof}
For $s\leq t$, the solution $y(s,t,\eta)$ of the system (\ref{nolin1}) is given by:
\begin{equation}
\label{ysolucion}
    y(s,t,\eta)=\Phi(s,t)\eta-\displaystyle\int_{s}^{t}\Phi(s,r)f(r,y(r,t,\eta))dr.
\end{equation}

By using (\ref{NUBG}) and \textbf{(P1)}, we can deduce that:
\begin{equation*}
    \begin{array}{rcl}
    |y(s,t,\eta)-y(s,t,\tilde{\eta})|&\leq& K_0 e^{a(t-s)+\varepsilon t}|\eta-\tilde{\eta}|
   \\\\
   &&+\displaystyle\int_{s}^{t}K_0e^{a(r-s)+\varepsilon r}L_{f}e^{-\theta r}|y(r,t,\eta)-y(r,t,\tilde{\eta})|dr,
    \end{array}
\end{equation*}
then, multiplying the previous inequality by $e^{as}$ and by using the Gronwall's Lemma, we obtain
\begin{equation*} 
    \begin{array}{rcl}
    e^{as}|y(s,t,\eta)-y(s,t,\tilde{\eta})|&\leq&\displaystyle K_0 e^{(a+\varepsilon) t}|\eta-\tilde{\eta}|e^{\int_{s}^{t}K_0L_{f}e^{(\varepsilon-\theta)r}dr},
    \end{array}
\end{equation*}
and by considering that $\theta-\varepsilon>0$, we have that
\begin{equation*} 
    \begin{array}{rcl}
        |y(s,t,\eta)-y(s,t,\tilde{\eta})|&\leq&K_0 e^{a(t-s)+\varepsilon t}|\eta-\tilde{\eta}|e^{\frac{K_0L_{f}}{\theta-\varepsilon}}.
    \end{array}
\end{equation*}

Similarly, if we consider $t\leq s$, we can ensure that 
\begin{equation*}
    |y(s,t,\eta)-y(s,t,\tilde{\eta})|\leq K_0 e^{a(s-t)+\varepsilon t}|\eta-\tilde{\eta}|e^{\frac{K_0L_{f}}{\theta-\varepsilon}},
\end{equation*}
and consequently, the estimates (\ref{ydiferencia1}) is satisfied.
\end{proof}

\subsection{Topological conjugacy}
% \textcolor{red}{
A careful revision of the literature concerned with nonuniform exponential dichotomy and applications shows the convenience of working with Banach spaces equipped with weighted norms whose weights are related with the dichotomy and the bounded growth constants.
% } 
In this context, given a fixed $b>0$ and a norm $|\cdot|$ on $\mathbb{R}^{n}$, we will consider the Banach space $(\mathcal{A},|\cdot|_{\mathcal{A}})$, where
\begin{displaymath}
\mathcal{A}:=\left\{\phi:\R_{0}^{+}\to\R^{n} \colon \textnormal{$\phi$ is continuous on $\mathbb{R}_{0}^{+}$ and} \,\,
\sup\limits_{t\geq0}e^{-bt}|\phi(t)|<\infty  \right\}
\end{displaymath}
and the norm is defined as follows:
$$
|\phi|_{\mathcal{A}}:=\sup\limits_{t\geq0}e^{-bt}|\phi(t)|.
$$

Based on all of the above, the first main result of this article is as follows.

\begin{theorem}
\label{Teo1}
Let $b>0$ fixed and assume that the linear system \eqref{lin1} has the properties of nonuniform exponential dichotomy on $\mathbb{R}_{0}^{+}$ with constants $(K,\alpha,\mu)$
and nonuniform bounded growth with constants $(K_{0},a,\varepsilon)$. Moreover, assume that the nonlinear perturbed system \eqref{nolin1} satisfies the properties \textnormal{\textbf{(P1)}--\textbf{(P2)}} with constants $L_{f}$, $\theta$ and $\delta$. If the above mentioned constants verify the additional conditions:
\begin{equation}
\label{conditions1}
\begin{array}{l}
    \delta+\mu< \min\{b,\alpha\}.\\
\end{array}
\end{equation}
\begin{equation}
\label{conditions2}
\left\{\begin{array}{l}
     \mu\leq \theta,\\
     \varepsilon<\theta,\\
     \alpha+\mu-\theta+b>0,\\
     \alpha-\mu+\theta-b>0,\\
\end{array}\right.
\end{equation}
\begin{equation}
\label{conditions3}
\begin{array}{l}
     \displaystyle\frac{KL_{f}}{\alpha+\mu-\theta+b}+\frac{KL_{f}}{\alpha-\mu+\theta-b}<1,
\end{array}
\end{equation}
then the systems \eqref{lin1} and \eqref{nolin1} are topologically conjugated. In addition, the homeomorphisms $H_{t}$ and $G_{t}$ are perturbations of the identity explicitely defined by
\begin{equation}
\label{homeos}
G(t,\eta):=\eta+w^{*}(t;(t,\eta))    
\quad \textnormal{and} \quad
H(t,\xi):=\xi+z^{*}(t;(t,\xi)),   
\end{equation}
where 
\begin{equation}
\label{w*}
\begin{array}{rcl}
w^*(t ;(\tau, \eta))&:=&\displaystyle-\int_0^{+\infty} \mathcal{G}(t,s) f(s, y(s, \tau, \eta)) d s
% &=&\displaystyle-\int_0^t \Phi(t, s)P(s) f(s, y(s, \tau, \eta)) d s+\int_{t}^{+\infty} \Phi(t, s)Q(s) f(s, y(s, \tau, \eta)) d s
\end{array}
\end{equation}
for some $(\tau, \eta) \in \R_{0}^{+} \times \mathbb{R}^{n}$ 
and, given $(\tau, \xi) \in \R_{0}^{+} \times \mathbb{R}^{n}$, the map $z^{*}(t;(\tau,\xi))$ corresponds to the fixed point of the map $\Gamma_{(\tau, \xi)}: \mathcal{A}\rightarrow \mathcal{A}$ defined by:
\begin{equation}
\label{Gamma}
\begin{array}{rcl}
\Gamma_{(\tau, \xi)} \phi(t)&:=&\displaystyle \int_0^{+\infty} \mathcal{G}(t,s) f(s, x(s, \tau, \xi)+\phi(s)) d s.
\end{array}
\end{equation} 
\end{theorem}

The proof of this result will be carried out in the next section. On the other hand, and despite the vast literature focused on linearization homeomorphisms, the above result deserves some remarks:

First of all, we stress that in previous works, the property of nonuniform bounded growth has not been assumed except in the works \cite{CRM2,Dragicevic}, where a related
property is considered. In addition, the majority of the previous works assume that
$t\mapsto \|A(t)\|$ is bounded, which is a particular case of uniform bounded growth.

The consideration of the nonuniformities in the dichotomy and bounded growth has some technical consequences:

\noindent $\bullet$ If $\max\{\mu,\varepsilon\}=\mu$, then (\ref{constraints-IH}) and (\ref{conditions1}) imply that $a>\delta$.

\noindent $\bullet$ The conditions (\ref{conditions2}) can be verified for any $\theta\geq \mu$ such that 
\begin{equation*}
%\label{reescrituratheta}
\max\{\mu+b-\alpha,\varepsilon\}<\theta<\mu+b+\alpha.    
\end{equation*}

\noindent $\bullet$ The condition (\ref{conditions3}) is equivalent to
\begin{displaymath}
2\alpha K L_{f} <(\alpha+\mu-\theta+b)(\alpha-\mu+\theta-b).   
\end{displaymath}

\noindent $\bullet$ 
% \textcolor{red}{
The parameters $\theta$, $b$ and $L_{f}$ are those that have a certain degree of freedom to be chosen, as long as they satisfy the conditions (\ref{conditions1}), (\ref{conditions2}) and (\ref{conditions3}).

Last but not least, a careful reading of the proof will show us that hypotheses (\ref{conditions1}), (\ref{conditions2}) and (\ref{conditions3}) imply that
\begin{equation}
\label{Green1}
\int_{0}^{\infty}\|\mathcal{G}(t,s)\|_{\mathcal{A}}e^{\delta s}\,ds < +\infty
\quad \textnormal{and} \quad
\int_{0}^{\infty}\|\mathcal{G}(t,s)\|_{\mathcal{A}}e^{-\theta s}\,ds < 1,
\end{equation}
which are reminiscent to (\ref{GVNL}), by considering $\|\cdot\|_{\mathcal{A}}$ as the operator norm induced by $|\cdot|_{\mathcal{A}}$. Nevertheless,
we emphasize a strong difference: we are working with estimations defined in terms with 
the nonuniform bounded growth constants and we are working with a norm arising from a ``parameterized" Banach space.

% }
% \end{itemize}

\section{Proof of Theorem \ref{Teo1}}

The construction of the homeomorphisms based in the exponential dichotomy essentially follows the lines of the Palmer's construction. Nevertheless, the explicit assumptions on the dichotomy and bounded growth properties induced technical differences.

The subsections 3.1 and 3.2 respectively describe properties of the maps $w^{*}$ and $z^{*}$ which will be useful to prove the topological conjugacy properties. The condition \textbf{(TC2)} is verified in the subsection 3.3  while \textbf{(TC3)} is verified in subsection 3.4.  Finally, the property
\textbf{(TC1)} will be verified in two subsections: the bijectivity in 3.5 and the continuity in 3.6.

\subsection{About the map $(t,\tau,\eta)\mapsto w^{*}(t;(\tau,\eta))$.} As before, let $t \mapsto y(t, \tau, \eta)$ be solution of (\ref{nolin1}) passing through $\eta$ at $t=\tau$. By using variation of parameters combined with the Green's operator defined in (\ref{Green}), it is straightforward to verify that (\ref{w*}) is the unique solution of the initial value problem:
\begin{equation*}
% \label{sistemaw}
\left\{\begin{aligned} \dot{w} &=A(t) w-f(t, y(t, \tau, \eta)) \\ w(0) &=\int_{0}^{+\infty} \Phi(0, s)Q(s) f(s, y(s, \tau, \eta)) d s. \end{aligned}\right.
\end{equation*}

It is important to emphasize that $w(0)$ is well defined due to the exponential dichotomy property with $0\leq s$ combined with \textbf{(P2)} and (\ref{conditions1}). In addition, the identity 
$y(s,r,y(r,\tau,\eta))=y(s,\tau,\eta)$ for any $s,\; r,\; \tau\geq 0$ is ensured for the uniqueness of the solutions and besides
implies the property
\begin{equation*}
% \label{w*propiedad}
w^*(t ;(\tau, \eta))=w^*(t ;(r, y(r, \tau, \eta))) \quad \text { for any } r \geq 0. 
\end{equation*}

\subsection{About the map  $\Gamma_{(\tau, \xi)}: \mathcal{A}\rightarrow \mathcal{A}$.} We can see that the map $\Gamma_{(\tau,\xi)}$ described by (\ref{Gamma}) can be written as follows:
\begin{equation*}
\begin{array}{rcl}
\Gamma_{(\tau, \xi)} \phi(t)&=&\displaystyle\int_0^t \Phi(t, s)P(s) f(s, x(s, \tau, \xi)+\phi(s)) \,ds\\\\
&& \displaystyle -\int_t^{+\infty} \hspace{-0.4cm}\Phi(t, s)Q(s) f(s, x(s, \tau, \xi)+\phi(s)) d s\\\\
&=&\Gamma_{+} \phi(t)+\Gamma_{-} \phi(t).
\end{array}
\end{equation*} 

Now, by using the exponential dichotomy estimates combined with \textbf{(P2)} and (\ref{conditions1}), we can deduce that for any $t\geq0$:
\begin{equation*}
\begin{array}{rcl}
|\Gamma_{+} \phi(t)|&\leq&\displaystyle\int_0^t Ke^{-\alpha(t-s)+\mu s}Me^{\delta s} d s\leq KMe^{-\alpha t}\displaystyle\int_0^t e^{(\alpha+\delta+\mu)s}\,ds,
\end{array}
\end{equation*}
which leads to
\begin{equation*}
\begin{array}{rcl}
e^{-bt}|\Gamma_{+} \phi(t)|&\leq&\displaystyle\frac{KM}{\alpha+\delta+\mu}e^{(\delta+\mu-b)t}.
\end{array}
\end{equation*} 
% \todo{$\delta+\mu-b<0$}

Similarly, we obtain that for any $t\geq0$:
\begin{equation*}
\begin{array}{rcl}
|\Gamma_{-} \phi(t)|&\leq&\displaystyle\int_t^{+\infty} Ke^{\alpha(t-s)+\mu s}Me^{\delta s} d s\leq KMe^{\alpha t}\displaystyle\int_t^{+\infty} e^{(-\alpha+\delta+\mu)s}\, ds,
\end{array}
\end{equation*}
and we have
\begin{equation*}
\begin{array}{rcl}
e^{-bt}|\Gamma_{-} \phi(t)|&\leq&\displaystyle\frac{KM}{\alpha-(\delta+\mu)}e^{(\delta+\mu-b)t}.
\end{array}
\end{equation*} 
% \todo{$\delta+\mu-b<0$}

Now, by considering (\ref{conditions1}), we conclude that the operator $\Gamma_{(\tau,\xi)}$ is well defined due that:
\begin{equation}
\label{Operador_A}
\begin{array}{rcl}
|\Gamma_{(\tau,\xi)} \phi|_{\mathcal{A}}&\leq&\displaystyle\frac{KM}{\alpha+\delta+\mu}+\frac{KM}{\alpha-(\delta+\mu)}<+\infty.
\end{array}
\end{equation}

Additionally, the operator is a contraction. In fact, by using the estimates of the exponential dichotomy, \textbf{(P1)} and the last two condition of (\ref{conditions2}), allow us to deduce that:
\begin{equation*}
\begin{array}{rcl}
|\Gamma_{(\tau,\xi)}\phi_1(t)-\Gamma_{(\tau,\xi)}\phi_2(t)|&\leq&|\Gamma_{+} \phi_1(t)-\Gamma_{+} \phi_2(t)|+|\Gamma_{-} \phi_1(t)-\Gamma_{-} \phi_2(t)|\\
&\leq&\displaystyle\int_0^{t} Ke^{-\alpha(t-s)+\mu s}L_{f}e^{-\theta s}|\phi_1(s)-\phi_2(s)| d s\\
&+&\displaystyle\int_{t}^{+\infty} Ke^{\alpha(t-s)+\mu s}L_{f}e^{-\theta s}|\phi_1(s)-\phi_2(s)| d s\\
&\leq& KL_{f}e^{-\alpha t}\displaystyle\int_0^{t} e^{(\alpha+\mu-\theta+b)s}e^{-bs}|\phi_1(s)-\phi_2(s)| d s\\
&+& KL_{f}e^{\alpha t}\displaystyle\int_t^{+\infty} e^{(-\alpha+\mu-\theta+b)s}e^{-bs}|\phi_1(s)-\phi_2(s)| d s\\
&\leq&\displaystyle\frac{KL_{f}}{\alpha+\mu-\theta+b}e^{(\mu-\theta+b)t}|\phi_1-\phi_2|_{\mathcal{A}}\\
&+&\displaystyle\frac{KL_{f}}{\alpha-\mu+\theta-b}e^{(\mu-\theta+b)t}|\phi_1-\phi_2|_{\mathcal{A}},\\
\end{array}
\end{equation*}
% \todo{$\alpha+\mu-\theta+b>0$ and $\mu-\theta\leq0$}
thus, from the definition of the norm $|\cdot|_{\mathcal{A}}$ and (\ref{conditions3}), we get the following estimation
\begin{equation*}
%\label{contraction}
    |\Gamma_{(\tau,\xi)}\phi_1-\Gamma_{(\tau,\xi)}\phi_2|_{\mathcal{A}}\leq\displaystyle\left [\frac{KL_{f}}{\alpha+\mu-\theta+b}+\frac{KL_{f}}{\alpha-\mu+\theta-b}\right ]|\phi_1-\phi_2|_\mathcal{A}, 
\end{equation*}
and we can conclude that the operator $\Gamma_{(\tau,\xi)}$ is a contraction. By using the Banach fixed point theorem it follows the existence and uniqueness of a fixed point satisfying the identity
\begin{equation}
\label{z*}
\begin{array}{rcl}
z^*(t ;(\tau, \xi))&=&\displaystyle\int_0^{+\infty} \mathcal{G}(t,s) f(s, x(s, \tau, \xi)+z^*(s ;(\tau, \xi))) d s.
\end{array}
\end{equation}

It is deserved to be mentioned that the above fixed point can be seen as the unique solution of the following 
% \textcolor{red}{
nonlocal
% } 
initial value problem:

\begin{equation} 
\label{sistemaz}
\left\{\begin{aligned} \dot{z} &=A(t) z+f(t, x(t, \tau, \xi)+z) \\ z(0) &= -\int_0^{+\infty} \Phi(0, s)Q(s) f(s, x(s, \tau, \xi)+z(s))\, ds. \end{aligned}\right.
\end{equation}

Moreover, due to the identity $\Phi(s,r)\Phi(r,\tau)=\Phi(s,\tau)$, for any $s,\; r,\; \tau\geq0$, we have that $x(s,r,x(r,\tau,\xi))=x(s,\tau,\xi)$ and the characterization of the fixed point as solution of the initial value problem (\ref{sistemaz}), implies the property
\begin{equation}
\label{z*propiedad}
z^*(t ;(\tau, \xi))=z^*(t ;(r, x(r, \tau, \xi))) \quad \text { for any } r \geq 0.
\end{equation}

% and on the other hand:
% \begin{equation*}
% \begin{array}{rcl}
% \|\Gamma_{-} \phi_1(t)-\Gamma_{-} \phi_2(t)\|&\leq&\displaystyle\int_{t}^{+\infty} Ke^{\alpha(t-s)+\mu s}L_{f}e^{-\theta s}\|\phi_1(s)-\phi_2(s)\| d s\\\\
% &\leq& KL_{f}e^{\alpha t}\displaystyle\int_t^{+\infty} e^{(-\alpha+\mu-\theta+b)s}e^{-bs}\|\phi_1(s)-\phi_2(s)\| d s,\\\\
% &\leq&\displaystyle\frac{KL_{f}}{\alpha-\mu+\theta-b}e^{(\mu-\theta+b)t}\|\phi_1-\phi_2\|_A,\\\\
% e^{-bt}\|\Gamma_{-} \phi_1(t)-\Gamma_{-} \phi_2(t)\|&\leq&\displaystyle\frac{KL_{f}}{\alpha-\mu+\theta-b}e^{(\mu-\theta)t}\|\phi_1-\phi_2\|_A.
% \end{array}
% \end{equation*}
% \todo{$\alpha-\{b-(\theta-\mu)\}>0$ y $\mu-\theta\leq0$}
% Therefore:
% \begin{equation*}
% e^{-bt}\|\Gamma \phi_1(t)-\Gamma \phi_2(t)\|\leq\displaystyle\left \{\frac{KL_{f}}{\alpha+\mu-\theta+b}+\frac{KL_{f}}{\alpha-\mu+\theta-b}\right\}e^{(\mu-\theta)t}\|\phi_1-\phi_2\|_A.
% \end{equation*}
% \todo{se puede pedir que $\mu-\theta\leq0$}
% Since 
% \begin{equation*} 
% \begin{array}{rcl}
% \displaystyle\left \{\frac{KL_{f}}{\alpha+\mu-\theta+b}+\frac{KL_{f}}{\alpha-\mu+\theta-b}\right\}&<&1\\\\
% \displaystyle\left \{\frac{KL_{f}}{\alpha+\{b-(\theta-\mu)\}}+\frac{KL_{f}}{\alpha-\{b-(\theta-\mu)\}}\right\}&<&1,
% \end{array}
% \end{equation*}
% \todo{podemos pedir que $b-(\theta-\mu)\geq0$ y además debemos pedir que $\alpha-\{b-(\theta-\mu)\}>0$}
% and it is easy to see that the operator $\Gamma_{(\tau, \xi)}$ is a contraction 

\subsection{The property (TC2) is verified}

By using (\ref{homeos}), (\ref{Gamma}) and (\ref{z*propiedad}), we can see that
$$
\begin{array}{rcl}
H[t, x(t, \tau, \xi)] &=&\displaystyle x(t, \tau, \xi)+\int_0^{+\infty} \mathcal{G}(t,s) f(s, x\left(s, t, x(t, \tau, \xi)+z^*(s ;(t, x(t, \tau, \xi)))\right) d s\\\\
&=&\displaystyle x(t, \tau, \xi)+\int_0^{+\infty} \mathcal{G}(t,s) f(s, x(s, t, x(t, \tau, \xi)+z^*(s ;(\tau, \xi))\, ds,

\end{array}
$$
which is equivalent to alternative characterizations. If fact, as $z^{*}$ is a fixed point of (\ref{Gamma}), the above equation can be seen as
\begin{equation}
\label{CAH1}
H[t,x(t,\tau,\xi)]=x(t,\tau,\xi)+z^{*}(t;(\tau,\xi)).
\end{equation}

In addition, upon inserting (\ref{CAH1}) in the integral term of $H[t,x(t,\tau,\xi)]$ leads to
\begin{equation}
\label{CAH2}
H[t,x(t,\tau,\xi)]=x(t,\tau,\xi)+
\int_{0}^{\infty}\mathcal{G}(t,s)f(s,H[s,x(s,\tau,\xi)])\;ds.
\end{equation}

By derivating (\ref{CAH1}) we can deduce the identity
$$
\begin{array}{rcl}
\displaystyle\frac{\partial}{\partial t} H[t, x(t, \tau, \xi)] &=&\displaystyle \frac{\partial}{\partial t} x(t, \tau, \xi)+\frac{\partial}{\partial t} z^*(t ;(\tau, \xi)) \\\\
&=&\displaystyle A(t) x(t, \tau, \xi)+A(t) z^*(t ;(\tau, \xi))+f(t, H[t, x(t, \tau, \xi)]) \\\\
&=&\displaystyle A(t) H[t, x(t, \tau, \xi)]+f(t, H[t, x(t, \tau, \xi)]),
\end{array}
$$
then $t \mapsto H[t, x(t, \tau, \xi)]$ is solution of (\ref{nolin1}) passing through $H(\tau, \xi)$ at $t=\tau$. As consequence of uniqueness of solutions we obtain the identity (\ref{conjugacion}).

Similarly, it can be proved that $t \mapsto G[t, y(t, \tau, \eta)]$ is solution of (\ref{lin1}) passing through $G(\tau, \eta)$ at $t=\tau$ and (\ref{conj2}) is verified.
% \begin{equation*}
% \label{Gconjugacion}
% G[t, y(t, \tau, \eta)]=x(t, \tau, G(\tau, \eta))=\Phi(t, \tau) G(\tau, \eta).
% \end{equation*}

\subsection{The property (TC3) is verified.}
It follows from the definition of $H$ that
\begin{equation*}
\begin{array}{rcl}
\displaystyle\sup_{t\geq0}e^{-bt}|H(t, \xi)-\xi|&=&|z^{*}|_{\mathcal{A}}<+\infty.
\end{array}
\end{equation*}

A similar inequality can be obtained for $|G(t, \eta)-\eta|$ since if we follow the lines of the estimation (\ref{Operador_A}), then we can conclude that $|w^{*}|_{\mathcal{A}}<+\infty$, and therefore 
\begin{equation*}
\begin{array}{rcl}
\displaystyle\sup_{t\geq0}e^{-bt}|G(t, \eta)-\eta|&=&|w^{*}|_{\mathcal{A}}<+\infty.
\end{array}
\end{equation*}

% \begin{equation*}
%     \begin{array}{rcl}
%     |w^*(t ;(\tau, \eta))|&\leq&\displaystyle\int_0^t \|\Phi(t, s)P(s) f(s, y(s, t, \eta))\| d s+\int_{t}^{+\infty} \|\Phi(t, s)Q(s) f(s, y(s, t, \eta))\| d s,\\
%     &\leq& \displaystyle\int_0^t Ke^{-\alpha(t-s)+\mu s}Me^{\delta s}ds+\int_{t}^{+\infty}Ke^{\alpha(t-s)+\mu s}Me^{\delta s}ds\\\\
%     &\leq& KMe^{-\alpha t}\displaystyle\int_0^t e^{(\alpha+\mu+\delta)s}ds+KMe^{\alpha t}\int_{t}^{+\infty}e^{-(\alpha-(\mu+\delta))s} ds\\\\
%     &\leq&\displaystyle KM\frac{1}{\alpha+\mu+\delta}e^{(\mu+\delta)t}+KM\frac{1}{\alpha-(\mu+\delta)}e^{(\mu+\delta)t}\\\\
%     \|w^{*}\|_A&\leq&\displaystyle\sup_{t\geq0}\left [\frac{KM}{\alpha+\mu+\delta}+\frac{KM}{\alpha-(\mu+\delta)}\right ]e^{(\mu+\delta-b)t}<+\infty.
%     \end{array}
% \end{equation*}

\subsection{The property (TC1) is verified: Bijectivity.}
Firstly, we will show that $H$ is bijective for any $t \geq 0$. We will start by showing that $H(t, G(t, \eta))=\eta$, for any $\eta\in\mathbb{R}^{n}$ and $t \geq 0$.
Indeed, by using the identity
$$
G[t,y(t,\tau,\eta)]=x(t,\tau,G(\tau,\eta)),
$$
verified in the previous subsection combined with (\ref{CAH2}) allows to see that
\begin{displaymath}
 \begin{array}{rcl}
H[t,G[t,y(t,\tau,\eta)]]&=&H[t,x(t,\tau,G(t,\eta)] \\\\
&=& \displaystyle x(t,\tau,G(\tau,\eta))+\int_{0}^{+\infty}\mathcal{G}(t,s)f(s,H[s,x(s,\tau,G(\tau,\eta))])\,ds \\\\
&=& \displaystyle 
G[t,y(t,\tau,\eta)]+\int_{0}^{+\infty}\mathcal{G}(t,s)f(s,H[s,G[s,y(s,\tau,\eta))])\,ds.\end{array}
\end{displaymath}

Now, by using the definition of $G$ stated in (\ref{homeos}) together with (\ref{w*}), we can see that the above identity becomes
\begin{displaymath}
\begin{array}{rcl}
H[t, G[t, y(t, \tau, \eta)]]&=& y(t, \tau, \eta)\\\\
&+&\displaystyle\int_0^{+\infty}\mathcal{G}(t,s)\{f\left(s, H[s,G[s,y(s,\tau,\eta)]))\right)-f(s,y(s,\tau,\eta))\}\,ds.
\end{array}
\end{displaymath}

Let $\omega(t)=H[t, G[t, y(t, \tau, \eta)]]-y(t, \tau, \eta)$. Then, a consequence of dichotomy properties and \textbf{(P1)}, we have that
$$
\begin{array}{rcl}
|\omega(t)| & \leq &\displaystyle\left | \int_0^t \Phi(t, s)P(s)\left\{f\left(s, H[s,G[s,y(s,\tau, \eta)]\right)-f(s, y(s, \tau, \eta))\right\}\;ds\right | \\\\
&+&\displaystyle\left | \int_t^{+\infty} \Phi(t, s)Q(s)\left\{f\left(s,H[s,G[s,y(s,\tau, \eta)]\right)-f(s, y(s, \tau, \eta))\right\}\;ds\right | \\\\
& \leq& \displaystyle \int_0^t Ke^{-\alpha(t-s)+\mu s}L_{f}e^{-\theta s}|\omega(s)|\, ds +   \int_t^{+\infty} Ke^{\alpha(t-s)+\mu s}L_{f}e^{-\theta s}|\omega(s)|\, ds.
\end{array}
$$

Therefore, we obtain the following estimates:
$$
\begin{array}{rcl}
|\omega(t)| &\leq& K L_{f} \displaystyle\int_0^t e^{-\alpha(t-s)+\mu s} e^{-\theta s}|\omega(s)| d s + K L_{f} \displaystyle\int_t^{+\infty} e^{\alpha(t-s)+\mu s} e^{-\theta s}|\omega(s)| d s\\
&\leq& K L_{f} e^{-\alpha t}\displaystyle\int_0^t e^{(\alpha+\mu-\theta+b) s} |\omega|_{\mathcal{A}} d s + K L_{f}e^{\alpha t} \displaystyle\int_t^{+\infty} e^{-(\alpha-\mu+\theta-b) s} |\omega|_{\mathcal{A}} d s,
\end{array}
$$
and multiplying this inequality by $e^{-bt}$ and taking supremum over $t\geq0$ we have that
$$
\begin{array}{rcl}
e^{-bt}|\omega(t)|&\leq&\displaystyle\left \{\frac{KL_{f}}{\alpha+\mu-\theta+b}+\frac{KL_{f}}{\alpha-\mu+\theta-b}\right\}e^{-(\theta-\mu)t}|\omega|_{\mathcal{A}}\\\\
|\omega|_{\mathcal{A}}&\leq&\displaystyle\left \{\frac{KL_{f}}{\alpha+\mu-\theta+b}+\frac{KL_{f}}{\alpha-\mu+\theta-b}\right\}|\omega|_{\mathcal{A}}.
\end{array}
$$
% \todo{$\alpha+\mu-\theta+b>0$ y $\alpha-\mu+\theta-b>0$}

Based on (\ref{conditions3}), it follows that $|\omega(t)|=|H[t, G[t, y(t, \tau, \eta)]]-y(t, \tau, \eta)|=0$ for any $t \geq 0$. In particular, when we take $t=\tau$, we obtain $H(\tau, G(\tau, \eta))=\eta$.

\bigskip

Now, we will prove that $G(t, H(t, \xi))=\xi$, for any $\xi\in\mathbb{R}^{n}$ and $t\geq0$. In fact, based on the definition of $G$ stated in (\ref{homeos}), the equation (\ref{CAH2}) and the identities
$$
y(s,t,H[t,x(t,\tau,\xi)])=y(s,t,y(t,\tau,H(\tau,\xi)))=y(s,\tau,H(\tau,\xi))
=H[s,x(s,\tau,\xi)],
$$
which are consequence of the property
(\ref{conjugacion}) that was verified in the subsection 3.3, we can deduce 
that
$$
\begin{array}{rcl}
G[t, H[t, x(t, \tau, \xi)]]&=& H[t, x(t, \tau, \xi)] -\displaystyle\int_0^{+\infty} \mathcal{G}(t,s)f(s, y(s, t, H[t, x(t, \tau, \xi)])) d s\\
% &+&\displaystyle\int_{t}^{+\infty} \Phi(t, s)Q(s) f(s, y(s, t, H[t, x(t, \tau, \xi)])) d s\\\\
&=& x(t, \tau, \xi)+\displaystyle\int_0^{+\infty} \mathcal{G}(t,s) f\left(s, H[s, x(s, \tau, \xi)])\right) d s\\
& &\displaystyle -\int_0^{+\infty} \mathcal{G}(t,s) f\left(s, H[s, x(s, \tau, \xi)])\right) d s\\
% &-&\displaystyle\int_0^t \Phi(t, s)P(s) f(s, y(s, t, H[t, x(t, \tau, \xi)])) d s\\\\
% &+&\displaystyle\int_{t}^{+\infty} \Phi(t, s)Q(s) f(s, y(s, t, H[t, x(t, \tau, \xi)])) d s\\\\
%&+& \displaystyle\int_0^t \Phi(t, s)\{f(s, H[s, x(s, \tau, \xi)])-f(s, y(s, \tau, H(\tau, \xi)))\} d s \\\\
&=& x(t, \tau, \xi)
\end{array}
$$
and taking $t=\tau$ leads to $G(\tau,H(\tau,\xi))=\xi$. As consequence, for any $t\geq0$, $H_{t}$ is a bijection with $G_{t}$ as its inverse.

\subsection{The property (TC1) is verified: Continuity.}
Let us recall that, for any fixed $t\geq 0$, the maps $\eta \mapsto G(t,\eta)$
and $\xi \mapsto H(t,\xi)$ can be seen as perturbations of the identity. In consequence, we only have to prove that $\xi \mapsto z^{*}(t;(t,\xi))$ and $\eta \mapsto w^{*}(t;(t,\eta))$ are continuous for any fixed $t\geq 0$.
\medskip

\noindent\textbf{$\bullet$ Continuity of the map $\xi \mapsto z^{*}(t;(t,\xi))$.}

If we consider a sequence $\{\xi_{n}\}_{n\in\mathbb{N}}\subset\R^{n}$ such that $\xi_{n}\to\xi$ when $n\to+\infty$, for any fixed $t\geq0$ and $\phi\in\mathcal{A}$, we define the sequence of functions
\begin{equation*}
    \ell_{n}(s):=\mathcal{G}(t,s)f(s,x(s,t,\xi_n)+\phi(s)).
\end{equation*}

It is straightforward to prove that, due that the continuity of the maps $\varsigma\mapsto f(s,\varsigma)$ and $\varsigma\mapsto x(s,t,\varsigma)$, for any fixed $s,\; t\geq0$, we have that
$$
\begin{array}{rcl}
\displaystyle \lim_{n\to+\infty}\ell_{n}(s)&=&\displaystyle\lim_{n\to+\infty} \mathcal{G}(t,s)f(s,x(s,t,\xi_n)+\phi(s))\\
&=&\mathcal{G}(t,s)f(s,x(s,t,\xi)+\phi(s)).
\end{array}
$$

Now we need to prove that $s\mapsto\mathcal{G}(t,s)f(s,x(s,t,\xi_n)+\phi(s))\in (L^{1}([0,+\infty)),|\cdot|_{\mathcal{A}})$. Indeed, by using the nonuniform exponential dichotomy, \textbf{(P2)} and (\ref{conditions1}), leads to the following estimation:
\begin{equation*}
\begin{array}{rcl}    |\Gamma_{(t,\xi_{n})}\phi(p)|&\leq&\displaystyle\int_{0}^{+\infty}\left \|\mathcal{G}(t,s)f(s,x(s,t,\xi_n)+\phi(s))\right \|ds\\
    &\leq& \displaystyle\int_{0}^{t}Ke^{-\alpha(t-s)+\mu s} Me^{\delta s}ds+\displaystyle\int_{t}^{+\infty}Ke^{\alpha(t-s)+\mu s} Me^{\delta s}ds\\
    &\leq&\displaystyle KMe^{-\alpha t}\frac{1}{\alpha+\mu+\delta}\left [e^{t(\alpha+\mu+\delta)}-1\right ]\\
    &&\displaystyle +KMe^{\alpha t}\frac{1}{\alpha-(\mu+\delta)}e^{-(\alpha-(\mu+\delta))t}\\
    &\leq&\displaystyle KM\left [ \frac{1}{\alpha+\mu+\delta}+\frac{1}{\alpha-(\mu+\delta)}\right ]e^{(\mu+\delta)t},\\    
    \end{array}
\end{equation*}
then we have that
\begin{equation*}
|\Gamma_{(t,\xi_{n})}\phi|_{\mathcal{A}}\leq\displaystyle\sup_{t\geq0}KM\left [ \frac{1}{\alpha+\mu+\delta}+\frac{1}{\alpha-(\mu+\delta)}\right ]e^{(\mu+\delta-b)t}<+\infty.
\end{equation*}

Therefore, for any fixed $t\geq0$ and by considering the norm $|\cdot|_\mathcal{A}$, the Dominated convergence theorem ensures that:
\begin{equation*}
\displaystyle\lim_{n\to+\infty}\Gamma_{(t,\xi_n)} \phi(p)=\Gamma_{(t,\xi)} \phi(p).
\end{equation*}

This means that given $\rho>0$, there exists $N\in\mathbb{N}$ such that for any $n\geq N$ we have that
\begin{equation*}
    \begin{array}{rcl}
    |\Gamma_{(t,\xi_n)} \phi-\Gamma_{(t,\xi)} \phi|_{\mathcal{A}}&<&\rho,\\
    \end{array}
\end{equation*}
thus, for any $t\geq0$ fixed, we can see that
$$\begin{array}{rcl}
|\Gamma_{(t,\xi_n)} \phi(p)-\Gamma_{(t,\xi)} \phi(p)|&=& e^{bt}e^{-bt}|\Gamma_{(t,\xi_n)} \phi-\Gamma_{(t,\xi)}\phi|\\
&\leq&e^{bt}|\Gamma_{(t,\xi_n)} \phi-\Gamma_{(t,\xi)}\phi|_{\mathcal{A}}<\rho e^{bt}.
\end{array}$$

Then, by using (\ref{z*}) we can consider $z^{*}(t,(t,\xi_n))$ and $z^{*}(t,(t,\xi))$ as the respective fixed points of $\Gamma_{(t,\xi_n)}$ and $\Gamma_{(t,\xi)}$, which allow to deduce that:
\begin{equation*}
    |z^{*}(t,(t,\xi_n))-z^{*}(t,(t,\xi))|<\tilde{\rho}=\rho e^{bt}
\end{equation*}
and the continuity follows.
\medskip

\noindent\textbf{$\bullet$ Continuity of the map $\eta \mapsto w^{*}(t,(t,\eta))$.}

\noindent We need to prove that 
$$\forall \rho>0,\; \exists \lambda(t,\rho)>0:\; |\eta-\tilde{\eta}|<\lambda\Rightarrow |w^{*}(t,(t,\eta))-w^{*}(t,(t,\tilde{\eta}))|<\rho.$$ 

In order to estimate $|w^{*}(t,(t,\eta))-w^{*}(t,(t,\tilde{\eta}))|$, it is straightforward to deduce that, for any $L>0$:
\begin{equation*}
    \begin{array}{rcl}
       |w^{*}(t,(t,\eta))-w^{*}(t,(t,\tilde{\eta}))|&\leq&\displaystyle\underbrace{\int_0^t |\Phi(t, s)P(s) \{f(s, y(s, \tau, \eta))-f(s, y(s, \tau, \tilde{\eta}))\}| d s}_{I_1(t)}\\\\
       &+&\displaystyle\underbrace{\int_{t}^{t+L} |\Phi(t, s)Q(s) \{f(s, y(s, \tau, \eta))-f(s, y(s, \tau, \tilde{\eta}))\}| d s}_{I_2(t)}\\\\
       &+&\displaystyle\underbrace{\int_{t+L}^{+\infty} |\Phi(t, s)Q(s) \{f(s, y(s, \tau, \eta))-f(s, y(s, \tau, \tilde{\eta}))\}| d s}_{I_3(t)}.
    \end{array}
\end{equation*}

For $I_1(t)$, by using the nonuniform exponential dichotomy estimates, \textbf{(P1)} and Lemma  \ref{continuidadCIy}, we obtain the following estimation: 

% Furthermore,\todo{$\theta>\varepsilon$ es en verdad necesaria en este caso?}
\begin{equation*}
    \begin{array}{rcl}
    I_1(t)&\leq&\displaystyle\int_0^t Ke^{-\alpha(t-s)+\mu s}L_{f}e^{-\theta s} |y(s, \tau, \eta)-y(s, \tau, \tilde{\eta})|ds\\
    &\leq&|\eta-\tilde{\eta}|KL_{f}K_0\displaystyle\int_0^t e^{-\alpha(t-s)+\mu s}e^{-\theta s}e^{a(t-s)+\varepsilon t}e^{\frac{K_0L_{f}}{\theta-\varepsilon}}ds\\
    &\leq&|\eta-\tilde{\eta}|\underbrace{e^{\frac{K_0L_f}{\theta-\varepsilon}}KL_{f}K_0e^{t(-\alpha+a+\varepsilon)}\displaystyle\int_0^t e^{(\alpha+\mu-\theta-a)s}ds}_{\psi_1(t)}.
    \end{array}
\end{equation*}

% For $I_2(t)$. For $t\leq s$, the solution $y(s,t,\eta)$ is given by:
% \begin{equation*}
%     y(s,t,\eta)=\Phi(s,t)\eta+\displaystyle\int_{t}^{s}\Phi(s,r)f(r,y(r,t,\eta))dr,
% \end{equation*}
% then we have the following estimation by using the Gronwall's Lemma:

% \begin{equation*}
%     \begin{array}{rcl}
%     |y(s,t,\eta)-y(s,t,\tilde{\eta})|&\leq& K_0 e^{a(s-t)+\varepsilon t}|\eta-\tilde{\eta}|+\displaystyle\int_{t}^{s}K_0e^{a(s-r)+\varepsilon r}L_{f}e^{-\theta r}|y(r,t,\eta)-y(r,t,\tilde{\eta})|dr\\\\
%     e^{-as}|y(s,t,\eta)-y(s,t,\tilde{\eta})|&\leq& K_0 e^{(-a+\varepsilon) t}|\eta-\tilde{\eta}|+\displaystyle\int_{t}^{s}K_0L_{f}e^{(\varepsilon-\theta)r}e^{-ar}|y(r,t,\eta)-y(r,t,\tilde{\eta})|dr\\\\
%     |y(s,t,\eta)-y(s,t,\tilde{\eta})|&\leq&K_0 e^{a(s-t)+\varepsilon t}|\eta-\tilde{\eta}|e^{\int_{t}^{s}K_0L_{f}e^{(\varepsilon-\theta)r}dr}
%     \end{array}
% \end{equation*}

In the similar way as before, we have the following estimate for $I_2(t)$: 
\begin{equation*}
    \begin{array}{rcl}    
    I_2(t)&\leq&|\eta-\tilde{\eta}|\underbrace{e^{\frac{K_0L_f}{\theta-\varepsilon}}KL_{f}K_0e^{t(\alpha-a+\varepsilon)}\displaystyle\int_t^{t+L} e^{(-\alpha+\mu-\theta+a)s}ds}_{\psi_2(t)}.
    \end{array}
\end{equation*}

On the other hand, by using dichotomy estimates, \textbf{(P2)} and (\ref{conditions1}), we obtain the following estimation for $I_3(t)$: 

\begin{equation*}
    \begin{array}{rcl}    I_3(t)&\leq&\displaystyle\int_{t+L}^{+\infty}Ke^{\alpha(t-s)+\mu s}2Me^{\delta s}ds=2KMe^{\alpha t}\displaystyle\int_{t+L}^{+\infty}e^{(-\alpha+\mu+\delta)s}ds\\
    &\leq&\displaystyle 2KMe^{\alpha t}\frac{1}{\alpha -(\mu+\delta)}e^{-(\alpha - (\mu+\delta))(t+L)}\\
    &\leq&\displaystyle 2KMe^{(\mu+\delta)t}\frac{1}{\alpha -(\mu+\delta)}e^{-(\alpha-(\mu+\delta))L}.
    \end{array}
\end{equation*}

Finally, for any fixed $t\geq0$, the condition (\ref{conditions1}) implies that for any $\rho>0$, there exists $L>0$ such that $I_3(t)\leq \frac{\rho}{3}$. Now, if we consider 
$$\lambda(t,\rho)=\frac{\rho}{3\max\{\psi_1(t),\psi_2(t)\}},$$
then the continuity follows.

\subsection{Alternative characterizations for $G$ and $H$}
In the work \cite{Castaneda3}, the authors have pointed out some alternative characterizations
for the homeomorphisms associated to the topological conjugacy in the discrete case, which can be easily generalized to the continuous framework.

In fact, notice that the identities (\ref{homeos}) combined with the bijectivity of $\eta \mapsto G(\tau,\eta)$ and $\xi \mapsto H(t,\xi)$ for any fixed $\tau\geq 0$
implies that:
\begin{equation*}
% \label{identite1}
\left\{\begin{array}{l}
w^{*}(\tau;(\tau,\eta))+z^{*}(\tau;(\tau,G(\tau,\eta)))=0\\\\
w^{*}(\tau;(\tau,H(\tau,\xi)))+z^{*}(\tau;(\tau,\xi))=0,
\end{array}\right.
\end{equation*}
which, combined again with the identities (\ref{homeos}) provides the alternative characterizations of $G$ and $H$ as fixed points:
\begin{equation*}
% \label{homeos-alt}
G(\tau,\eta)=\eta-z^{*}(\tau;(\tau,G(\tau,\eta))) \quad
\textnormal{and} \quad
H(\tau,\xi)=\xi-w^{*}(\tau;(\tau,H(\tau,\xi))).
\end{equation*}

Last but not least, another alternative characterization for the map $\eta \mapsto G(s,\eta)$
was provided also in \cite[Eq. 2.19]{Castaneda3} for the discrete case and extended
for the continuous case in \cite[Eq. 3.6] {Jara}. This characterization stands out for its great
simplicity, namely:
\begin{equation}
\label{caeq}
G(s,\eta)=\displaystyle\Phi(s,0)\left \{ y(0,s,\eta)+w^{*}(0;(s,\eta))\right\}.
\end{equation}

The characterization (\ref{caeq}) will be useful to address the study of differentiability of the topological congugacy, which is the topic of the next section.

\section{Smoothness of topological conjugacy}

A result of R. Plastock \cite[Cor.2.1]{Plastock}, which provides necessary and sufficient conditions ensuring
the existence of a $\mathbb{R}^{n}$--diffeormorphism, has been recurrently employed in order to deduce smoothness properties of the topological conjugacy between systems \eqref{lin1} and \eqref{nolin1}. Since the homeomorphism $\eta\mapsto G(t,\eta)$ defined by \eqref{homeos} is a perturbation of the identity, the analysis of the norm of its Jacobian matrix is indispensable to use this result. In this context, we invite the reader to see the works \cite{Castaneda3, CRM, CRM2, CT, Jara} to review the scope of this tool.

\subsection{Technical results and smoothness property of the topological conjugacy} The main focus of this section corresponds to the study of the smoothness properties for the homeomorphism $\eta \mapsto G(t,\eta)$, which sends solutions of system (\ref{lin1}) into the system (\ref{nolin1}).
Based on this, we will consider the following additional hypothesis:

\begin{itemize}
    \item [\textbf{(P3)}]
    The function $x\mapsto f(t,x)$ and its partial derivative with respect to $x$, denoted by $\partial_2 f$, are continuous with respect to $(t,x)$.
\end{itemize}

\begin{remark}
    The condition \textnormal{\textbf{(P1)}} combined with \textnormal{\textbf{(P3)}} implies that, for any fixed $t\geq0$ and $x\in\mathbb{R}^{n}$, we have the estimation:
    \begin{equation}
     \label{cotaderivada}
     \|\partial_2 f(t,x)\|=\sup\limits_{\eta \in \mathbb{R}^{n}, |\eta|=1}\frac{\left|\partial_2 f(t,x)\eta\right|}{|\eta|}\leq L_{f}e^{-\theta t}.
    \end{equation}
\end{remark}

Some technical results that will pave the way to the second main result of this paper will be stated and proved below.

\begin{lemma}\label{estimacionderivaY}
If the linear system \eqref{lin1} admits nonuniform bounded growth with constants $(K_0, a, \varepsilon)$, the condition  \textnormal{\textbf{(P1)}} is satisfied with $\varepsilon<\theta$ then, for any $s,t\geq0$, the solution of \textnormal{(\ref{nolin1})} $t\mapsto y(t,\tau,\eta)$ verifies:
\begin{equation}
\label{yparcial}
    \displaystyle \left\|\frac{\partial y}{\partial \eta}(s,t,\eta)\right\| \leq\displaystyle K_0 e^{a|s-t|+\varepsilon t} e^{\frac{L_f K_0}{\theta-\varepsilon}}.
\end{equation}    
\end{lemma}

\begin{proof}
From the solution given by the equation (\ref{ysolucion}), for $s\leq t$ we have that

$$\displaystyle \frac{\partial y}{\partial \eta}(s, t, \eta)=\displaystyle\Phi(s, t)-\int_s^t \Phi(s, r) \partial_2 f(r, y(r, t, \eta)) \frac{\partial y}{\partial \eta}(r, t, \eta) d r.$$

By considering the nonuniform bounded growth property described by \eqref{NUBG} and the inequality \eqref{cotaderivada}, we obtain the following estimation:

$$\displaystyle \left \|\frac{\partial y }{\partial \eta}(s, t, \eta)\right \|\leq\displaystyle K_0 e^{a(t-s)+\varepsilon t}+\int_s^t K_0 e^{a(r-s)+\varepsilon r } L_{f} e^{-\theta r}\left \| \frac{\partial y}{\partial \eta}(r, t, \eta)\right \|  d r,$$
then by multiplying the above expression by $e^{as}$ and by using the Gronwall's Lemma we deduce that 

$$\displaystyle e^{a s} \left \|\frac{\partial y }{\partial \eta}(s, t, \eta)\right \|\leq\displaystyle K_0 e^{(a+\varepsilon) t} e^{\frac{L_f K_0}{\theta-\varepsilon}},$$
and the result follows since for the case $t \leq s$ we can do it in a similar way.
\end{proof}

\begin{lemma}
\label{Gdif}
Assume that the linear system \eqref{lin1} has the properties of nonuniform exponential dichotomy 
and nonuniform bounded growth with constants $(K, \alpha, \mu)$ and $(K_0, a, \varepsilon)$, respectively. Moreover, assume that the nonlinear system \eqref{nolin1} has a perturbation $f$ satisfying the properties \textnormal{\textbf{(P1)}},\textnormal{\textbf{(P2)}} and \textnormal{\textbf{(P3)}}. If the dichotomy and bounded growth constants verify \eqref{conditions1}, \eqref{conditions2}, \eqref{conditions3} and the additional condition
\begin{equation}\label{condition4}
    \alpha-\mu+\theta-a>0,
\end{equation}
then the topological conjugacy between the systems \eqref{lin1} and \eqref{nolin1} has the properties that for any fixed $t\geq 0$, the map $\eta \mapsto G(t,\eta)$ is differentiable.
\end{lemma}

\begin{proof}
    
In order to show that the map $\eta \mapsto G(t,\eta)$ is differentiable, it will be useful to work with the alternative characterization stated in (\ref{caeq}), namely 
\begin{equation*}
G(s,\eta)=\displaystyle\Phi(s,0)\left \{ y(0,s,\eta)+w^{*}(0;(s,\eta))\right\},
\end{equation*}
which implies that it is enough to show that the map $\eta \mapsto w^{*}(0,(s, \eta))$ is differentiable, since the condition \textbf{(P3)} implies that the map $\eta \mapsto y(t,s, \eta)$ is differentiable for any $(t,s)$ and we refer to \cite[Ch. V]{Hartman}  
and \cite[Th.6.1, p.89]{Sideris} for details. 

Firstly, by considering \eqref{w*}, we will prove that for any $s\in\mathbb{R}_{0}^{+}$, $\eta\in\mathbb{R}^{n}$, the following linear map is well defined:
$$
Dw^{*}(0;(s,\eta))=-\displaystyle\int_0^{+\infty} \mathcal{G}(0, r) \partial_{2} f(r, y(r, s, \eta)) \frac{\partial y}{\partial \eta}(r, s, \eta)\,dr.
$$

In fact, for any fixed $s\in\mathbb{R}_0^{+}$, $\eta\in\mathbb{R}^{n}$ and by using \eqref{NUED}, \eqref{cotaderivada}, \eqref{yparcial} and \eqref{condition4}, we ensure that
$$\begin{array}{rcl}
\left \|Dw^{*}(0;(s,\eta)) \right \|&\leq&\displaystyle\int_0^{s}Ke^{-\alpha r +\mu r}L_{f}e^{-\theta r}K_0 e^{a(s-r)+\varepsilon s} e^{\frac{L_f K_0}{\theta-\varepsilon}}\;dr\\
&+&\displaystyle\int_s^{+\infty}Ke^{-\alpha r +\mu r}L_{f}e^{-\theta r}K_0 e^{a(r-s)+\varepsilon s} e^{\frac{L_f K_0}{\theta-\varepsilon}}\;dr\\
&=&\displaystyle KK_0L_{f}e^{\frac{L_f K_0}{\theta-\varepsilon}}e^{(a+\varepsilon)s}\int_{0}^{s}e^{-(\alpha-\mu+\theta+a)r}\;dr\\
&+&\displaystyle KK_0L_{f}e^{\frac{L_f K_0}{\theta-\varepsilon}}e^{(-a+\varepsilon)s}\int_{s}^{+\infty}e^{-(\alpha-\mu+\theta-a)r}\;dr\\
&\leq& \displaystyle \frac{KK_0L_{f}e^{\frac{L_f K_0}{\theta-\varepsilon}}}{\alpha-\mu+\theta+a}e^{(a+\varepsilon)s}\left(1-e^{-(\alpha-\mu+\theta+a)s}\right)\\
&+&\displaystyle \frac{KK_0L_{f}e^{\frac{L_f K_0}{\theta-\varepsilon}}}{\alpha-\mu+\theta-a}e^{(-a+\varepsilon)s}e^{-(\alpha-\mu+\theta-a)s}<+\infty.
\end{array}$$

Secondly, we know that $\eta \mapsto w^{*}(0;(s,\eta))$ is differentiable if:
\begin{displaymath}
\lim\limits_{\delta\to 0}\frac{w^{*}(0;(s,\eta+\delta))-w^{*}(0;(s,\eta))-Dw^{*}(0;(s,\eta))}{|\delta|}=0,
\end{displaymath}
which is equivalent to:
\begin{equation*}
%\label{TCD}
\lim\limits_{j\to+\infty}\Psi(s,\eta)=\lim\limits_{j\to +\infty}\int_{0}^{\infty}\varphi_{j}(r)\,dr=0,
\end{equation*}
where $r\mapsto \varphi_{j}(r)$ is given by:
\begin{displaymath}
\varphi_{j}(r):=\mathcal{G}(0, r) \left(\frac{f\left(r, y\left(r, s, \eta+\delta_j\right)\right)-f(r, y(r, s, \eta))-\partial_{2} f(r, y(r, s, \eta)) \frac{\partial y}{\partial \eta}(r, s, \eta) \delta_j}{\left|\delta_j\right|}\right)
\end{displaymath}
and
\begin{displaymath}
\Psi(s,\eta):=\int_{0}^{\infty}\varphi_{j}(r)\,dr,
\end{displaymath}
while $\{\delta_{j}\}_{j}$ is a sequence of vectors different from zero which converges to the origin. Notice that, a direct consequence of the condition \textbf{(P3)}, is the pointwise convergence:
$$
\lim _{j\to +\infty} \varphi_j(r)=0. 
$$

%and for any sequence $(\delta_{j})_{j\in\mathbb{N}}$ properly convergent to zero, we have that 
%$$\displaystyle\lim_{j\to+\infty}\psi_{j}(s,\eta)=0,$$
%where
%\begin{equation}
%\label{w*diferenciabilidad}
%\psi_{j}(s,\eta)=\frac{w^{*}(0;(s,\eta+\delta_{j}))-w^{*}(0;(s,\eta))-Dw^{*}(0;%(s,\eta))\delta_{j}}{|\delta_{j}|},
%\end{equation}
%with $Dw^{*}(0;(s,\eta))$ representing the derivative of the function $\eta\mapsto w^{*}(0;(s,\eta))$ with respect to $\eta$. 

%On the other hand, for any fixed $s\in\mathbb{R}_0^{+}$ and $\eta\in\mathbb{R}^{n}$, we define %the sequence
%$$
%\varphi_j(r):=\mathcal{G}(0, r) \left(\frac{f\left(r, y\left(r, s, \eta+\delta_j\right)\right)-%f(r, y(r, s, \eta))-\partial_{2} f(r, y(r, s, \eta)) \frac{\partial y}{\partial \eta}(r, s, %\eta) \delta_j}{\left|\delta_j\right|}\right),
%$$

Now, our objective is prove that, by using the dominated convergence theorem, we have
\begin{equation}
\label{varphiTCD}
\displaystyle\lim_{j\to+\infty}\int_{0}^{+\infty}\varphi_{j}(r)\; dr=\displaystyle\int_{0}^{+\infty}\lim_{j\to+\infty}\varphi_{j}(r)\; dr=0.
\end{equation}

In order to achieve this, we will prove that the sequence $(\varphi_j)_{j\in\mathbb{N}}$ is dominated by an integrable function in the variable $r$. Indeed, by combining \textbf{(P1)}
and \eqref{cotaderivada}, we deduce that
\begin{equation}
\label{phi_j}
\begin{aligned}
\left|\varphi_j(r)\right|& \leq\|\mathcal{G}(0, r)\| \left(\frac{L_f e^{-\theta r}\left|y\left(r, s, \eta+\delta_j\right)-y(r, s, \eta)\right|+L_f e^{-\theta r}\left|\frac{\partial y}{\partial \eta}(r, s, \eta) \delta_j\right|}{\left|\delta_j\right|}\right) \\
& = \|\mathcal{G}(0, r)\| L_f e^{-\theta r}\left(\frac{\left|y\left(r, s, \eta+\delta_j\right)-y(r, s, \eta)\right|}{\left|\delta_j\right|}+\left\|\frac{\partial y}{\partial \eta}(r, s, \eta)\right\|\right) .
\end{aligned}
\end{equation}
\newpage
The following is a breakdown of our analysis in two cases:

\noindent \textbf{Case 1): $s\leq r$.} 

By using Lemma \ref{continuidadCIy} and for all $j\in\mathbb{N}$ we have that:
% combined with (\ref{condition4}), we have that 
$$
\begin{array}{rcl}
\displaystyle \frac{|y(r,s,\eta+\delta_j)-y(r,s,\eta)|}{|\delta_j|}&\leq& 
% \displaystyle K_0 e^{a(s-\tau)+\varepsilon \tau}e^{\int_{\tau}^{s}K_0L_{f}e^{(\varepsilon-\theta)r}dr},\\\\
\displaystyle K_0 e^{a(r-s)+\varepsilon s}e^{\frac{L_fK_0}{\theta-\varepsilon}}.
\end{array}
$$

 The above estimation combined with Lemma \ref{estimacionderivaY} leads to
$$
\begin{aligned}
\left|\varphi_j(r)\right| & \leq \|\mathcal{G}(0, r)\| L_f e^{-\theta r}\left(K_0 e^{a(r-s)+\varepsilon s}e^{\frac{L_{f}K_0}{\theta-\varepsilon}}+ \displaystyle K_0 e^{a(r-s)+\varepsilon s} e^{\frac{L_f K_0}{\theta-\varepsilon}}\right) \\
&\leq 2 Ke^{-\alpha r+\mu r} L_f e^{-\theta r} K_0 e^{a(r-s)+\varepsilon s}e^{\frac{L_{f}K_0}{\theta-\varepsilon}}.
\end{aligned}
$$

\noindent \textbf{Case 2): $r\leq s$.}
When we work on the compact set $[0, s]$ and recalling that the map $\eta \mapsto y(r, s, \eta)$ is a class $\mathcal{C}^{1}$ map, we can observe that the sequence of continuous functions 
$$r \mapsto \frac{\left|y\left(r, s, \eta+\delta_j\right)-y(r, s, \eta)\right|}{\left|\delta_j\right|} \textnormal{ converges pointwise to } r \mapsto\left\|\frac{\partial y}{\partial \eta}(r, s, \eta)\right\|$$ 
% converge pointwise to $s \mapsto\left\|\frac{\partial y}{\partial \eta}(s, \tau, \eta)\right\|$ 
when $j \rightarrow +\infty$, which is also a continuous function. Then, given $\rho>0$, for any fixed $r\in[0,s]$ there exist $\hat{j}:=\hat{j}(\rho,r) \in \mathbb{N}$ such that for any $j \geq \hat{j}$ we have that:
$$
\frac{\left|y\left(r, s, \eta+\delta_j\right)-y(r, s, \eta)\right|}{\left|\delta_j\right|}+\left\|\frac{\partial y}{\partial \eta}(r, s, \eta)\right\| \leq 2\left\|\frac{\partial y}{\partial \eta}(r, s, \eta)\right\|+\rho .
$$

Therefore, given $\rho=1$, for any fixed $r\in[0,s]$ and $j \geq \hat{j}=\hat{j}(1,r)$, by considering (\ref{yparcial}) we have the following estimation for (\ref{phi_j}):
$$
\begin{aligned}
\left|\varphi_j(r)\right| & \leq\|\mathcal{G}(0, r)\| L_f e^{-\theta r}\left(\frac{\left|y\left(r, s, \eta+\delta_j\right)-y(r, s, \eta)\right|}{\left|\delta_j\right|}+\left\|\frac{\partial y}{\partial \eta}(r, s, \eta)\right\|\right) \\
& \leq Ke^{-\alpha r+\mu r}L_f e^{-\theta r}\left(2\left\|\frac{\partial y}{\partial \eta}(r, s, \eta)\right\|+1\right)\\
&\leq Ke^{-\alpha r+\mu r}L_f e^{-\theta r} \left (2\displaystyle K_0 e^{a(s-r)+\varepsilon s} e^{\frac{L_fK_0}{\theta-\varepsilon}}+1\right ).
\end{aligned}
$$

Summarizing, for any fixed $s\in\mathbb{R}_0^{+}$ if we define $\mathcal{F}: \mathbb{R}_{0}^{+} \rightarrow \mathbb{R}$ by
$$
\mathcal{F}(r)=\left\{\begin{array}{cc}
KL_fe^{(-\alpha+\mu-\theta) r} \left (2\displaystyle K_0 e^{a(s-r)+\varepsilon s} e^{\frac{L_{f}K_0}{\theta-\varepsilon}}+1\right )  & s \geq r \geq 0, \\
2 KL_fK_0e^{(-\alpha +\mu-\theta )r} e^{a(r-s)+\varepsilon s}e^{\frac{L_{f}K_0}{\theta-\varepsilon}} & r > s,
\end{array}\right.
$$
therefore, we have prove that for every $j \geq \hat{j}(1,r)$, $\left|\varphi_j(r)\right| \leq \mathcal{F}(r)$. Based on the above, we will prove that $\mathcal{F}\in L^{1}(\mathbb{R}_{0}^{+})$. In fact, by considering (\ref{condition4}), we can deduce that for any fixed $s\in\mathbb{R}_0^{+}$:
$$
\begin{array}{rcl}
\displaystyle\int_0^{+\infty} \mathcal{F}(r) d r &= &\displaystyle e^{\frac{L_fK_0}{\theta-\varepsilon}}2KK_0L_f e^{(a+\varepsilon) s}\int_{0}^{s}e^{(-\alpha+\mu-\theta-a)r} dr+KL_f\int_{0}^{s}e^{(-\alpha+\mu-\theta)r} dr\\
&+& \displaystyle 2 KK_0L_fe^{\frac{L_{f}K_0}{\theta-\varepsilon}} e^{(-a+\varepsilon) s}\int_{s}^{+\infty} e^{(-\alpha +\mu-\theta+a )r} dr\\
&\leq&\displaystyle e^{\frac{L_fK_0}{\theta-\varepsilon}}2KK_0L_f e^{(a+\varepsilon)s}\frac{1}{\alpha-\mu+\theta+a} +KL_f \frac{1}{\alpha-\mu+\theta}\\
&+& \displaystyle e^{\frac{L_{f}K_0}{\theta-\varepsilon}}2KK_0L_f e^{(-a+\varepsilon)s}\frac{e^{-(\alpha-\mu+\theta-a)s}}{\alpha-\mu+\theta-a}\\
&\leq&\displaystyle \frac{e^{\frac{L_fK_0}{\theta-\varepsilon}}2KK_0L_f}{\alpha-\mu+\theta+a} e^{(a+\varepsilon)s} +\frac{KL_f}{\alpha-\mu+\theta}\\
&+& \displaystyle \frac{e^{\frac{L_{f}K_0}{\theta-\varepsilon}}2KK_0L_f}{\alpha-\mu+\theta-a} e^{-(\alpha-\mu+\theta-\varepsilon))s}<+\infty. 
\end{array}
$$

As consequence, by using dominated convergence theorem we can ensure that \eqref{varphiTCD} is satisfied. Therefore, based on the definitions of $\psi_{j}(s,\eta)$, $\varphi_{j}(r)$ and the equation \eqref{varphiTCD}, we have that
$$\displaystyle\lim_{j\to+\infty}\Psi_{j}(s,\eta)=\lim_{j\to+\infty}\int_{0}^{+\infty}\varphi_{j}(r)\;dr=\int_{0}^{+\infty}\lim_{j\to+\infty}\varphi_{j}(r)\;dr=0,$$
then $\eta \mapsto w^*(0 ;(s, \eta))$ is differentiable and the lemma follows.
\end{proof}

\begin{remark}
By using the definition of $\eta \mapsto G(t,\eta)$ described by \eqref{homeos} and by following the lines of the previous proof, we can prove that    

\begin{equation}
\label{JacoG}
\displaystyle \frac{\partial G}{\partial \eta}(t,\eta)=\displaystyle I-\underbrace{\int_{0}^{+\infty}\mathcal{G}(t,s)\partial_{2}f(s,y(s,t,\eta))\frac{\partial y}{\partial \eta}(s,t,\eta)ds}_{=: \Lambda (t,\eta)},
\end{equation}
which is well defined. 
\end{remark}

\subsection{Diffeomorphism result} By considering the term $\Lambda(t,\eta)$ defined in (\ref{JacoG}) and in order to obtain an estimation for any fixed $t\in\mathbb{R}^{+}_{0}$ of this term, the following theorem will guarantee under which conditions the function $\eta\mapsto G(t,\eta)$ is a diffeomorphism.

\begin{theorem}
\label{Teo2}
Assume that the linear system \eqref{lin1} has the properties of nonuniform exponential dichotomy 
and nonuniform bounded growth with constants $(K, \alpha, \mu)$ and $(K_0, a, \varepsilon)$, respectively. Moreover, assume that the nonlinear system \eqref{nolin1} has a perturbation $f$ satisfying the properties \textnormal{\textbf{(P1)}}, \textnormal{\textbf{(P2)}} and \textnormal{\textbf{(P3)}} such that \eqref{conditions1}, \eqref{conditions2}, \eqref{conditions3} are satisfied. If the additional estimations are satisfied:
\begin{equation}
    \label{thetacondition1}
        \max\{\mu+\alpha-a,\mu-(\alpha-a)\}<\theta,
    \end{equation}
    \begin{equation}
    \label{thetacondition2}
        \mu+\varepsilon\leq\theta,
    \end{equation}
    \begin{equation}
    \label{thetacondition3}
        KL_fK_0e^{\frac{L_f K_0}{\theta-\varepsilon}}\displaystyle\max\left\{\frac{1}{\theta-(\alpha-a)-\mu},\frac{1}{\theta+(\alpha-a)-\mu}\right\}\leq \frac{1}{c},
    \end{equation}
    with $c>2$, then there exists an interval $[0,\tilde{t}_{c})$ where $\eta\mapsto G(t,\eta)$ is a $C^{1}$-diffeomorphism for any $t\in[0,\tilde{t}_{c})$.

\end{theorem}
% }
\begin{proof}
Firstly, let us recall that $\eta \mapsto G(t,\eta)$ verifies the property \textbf{(TC3)}. Secondly, the parameters $\theta$ and $b$ verify 
(\ref{thetacondition1}), 
 (\ref{thetacondition2}) and 
 (\ref{thetacondition3}), which implies that the hypothesis of the Lemma \ref{Gdif} are satisfied, then we have that $\eta\mapsto G(t,\eta)$ is differentiable for any fixed $t\geq0$. 

On the other hand, we can combine Definition \ref{NUEDD}, (\ref{Green}), \eqref{cotaderivada} and the Lemma \ref{estimacionderivaY}, therefore we obtain the following: 
$$
\begin{array}{rcl}
\left \| \Lambda(t,\eta)\right \|&\leq&\displaystyle \int_{0}^{t}Ke^{-\alpha(t-s)+\mu s}L_{f}e^{-\theta s}K_0e^{a(t-s)+\varepsilon t}e^{\frac{L_{f}K_0}{\theta-\varepsilon}}ds\\\\
&&+\displaystyle\int_{t}^{+\infty}Ke^{\alpha(t-s)+\mu s}L_{f}e^{-\theta s}K_0e^{a(s-t)+\varepsilon t}e^{\frac{L_{f}K_0}{\theta-\varepsilon}}ds\\\\
&=& \displaystyle e^{\frac{L_{f}K_0}{\theta-\varepsilon}}KL_{f}K_0e^{(-\alpha+a+\varepsilon)t}\displaystyle\int_{0}^{t}e^{-(\theta+a-\alpha-\mu)s}ds\\\\
&&+e^{\frac{L_{f}K_0}{\theta-\varepsilon}}KL_{f}K_0 e^{(\alpha-a+\varepsilon)t}\displaystyle\int_{t}^{+\infty}e^{-(\alpha-\mu+\theta-a)s}ds,
\end{array}$$
and by considering (\ref{thetacondition1}), we have that:
$$\begin{array}{rcl}
\left \| \Lambda(t,\eta)\right \|&=&\displaystyle \frac{e^{\frac{L_{f}K_0}{\theta-\varepsilon}}KL_{f}K_0}{\theta+a-\alpha-\mu}e^{(-\alpha+a+\varepsilon)t}\displaystyle \left (1-e^{-(\theta+a-\alpha-\mu)t}\right )\\\\
&&+\displaystyle\frac{e^{\frac{L_{f}K_0}{\theta-\varepsilon}}KL_{f}K_0}{\alpha-\mu+\theta-a} e^{(\alpha-a+\varepsilon)t}\displaystyle e^{-(\alpha-\mu+\theta-a)t}\\\\
&\leq&\displaystyle \frac{e^{\frac{L_{f}K_0}{\theta-\varepsilon}}KL_{f}K_0}{\theta+a-\alpha-\mu}e^{(-\alpha+a+\varepsilon)t}+\displaystyle\frac{e^{\frac{L_{f}K_0}{\theta-\varepsilon}}KL_{f}K_0}{\alpha-\mu+\theta-a} \displaystyle e^{-(\theta-(\mu+\varepsilon))t}.
\end{array}$$
% then we conclude that:
% \begin{equation} 
% \label{Lambdacota}
% \|\Lambda(t,\eta)\|\leq\displaystyle \frac{e^{\frac{L_{f}K_0}{\theta-\varepsilon}}KL_{f}K_0}{\theta+a-\alpha-\mu}e^{(-\alpha+a+\varepsilon)t}+\displaystyle\frac{e^{\frac{L_{f}K_0}{\theta-\varepsilon}}KL_{f}K_0}{\alpha-\mu+\theta-a} \displaystyle e^{-(\theta-(\mu+\varepsilon))t}.
% \end{equation}

Additionally, the conditions (\ref{thetacondition1}), (\ref{thetacondition2}) and (\ref{thetacondition3}) allow us to ensure the following estimate:
\begin{equation*} 
%\label{CotaLambda}
\|\Lambda(t,\eta)\|\leq \displaystyle\frac{1}{c}\left(e^{(-\alpha+a+\varepsilon)t}+e^{-(\theta-(\mu+\varepsilon))t}\right)\leq\displaystyle\frac{1}{c}\left(e^{(-\alpha+a+\varepsilon)t}+1\right),
\end{equation*}
which implies that there exists  $\tilde{t}_{c}$ such that $\left \| \Lambda(t,\eta)\right \|<1$ for any $t\in[0,\tilde{t}_{c})$ and $\eta \in \mathbb{R}^{n}$, therefore $\det \frac{\partial G}{\partial \eta}(t,\eta)\neq 0$ for any $(t,\eta)\in [0,\tilde{t}_{c})\times \mathbb{R}^{n}$.

Finally, by using the continuity of
$\eta \mapsto \partial_{2}f(s,y(s,t,\eta))\frac{\partial y}{\partial \eta}(s,t,\eta)$, combined with the last estimate for any fixed $t\in[0,\tilde{t}_{c})$,
the Dominated convergence theorem implies that $\eta \mapsto \frac{\partial G}{\partial \eta}(t,\eta)$ is continuous in $\mathbb{R}^{n}$.

The result follows straightforwardly as a consequence of a result of Plastock  \cite[Cor.2.1]{Plastock}. 
\end{proof}

% \begin{corollary}
% \label{Teo2}
%     Let us consider the hypotheses of Lemma \ref{Gdif}, but instead of the parameters $\theta$ and $b$ we take into account the existence of a sequences $\{\theta_j\}_{j\in\mathbb{N}}$ and $\{b_j\}_{j\in\mathbb{N}}$ satisfying the conditions \eqref{conditions1}, \eqref{conditions2} and \eqref{conditions3}. In addition, if for the sequence $\{\theta_j\}_{j\in\mathbb{N}}$ we replace the condition \eqref{condition4} by the following ones: 
%     \begin{equation}
%     \label{thetacondition1}
%         \max\{\mu+\alpha-a,\mu-(\alpha-a)\}<\theta_j,
%     \end{equation}
%     \begin{equation}
%     \label{thetacondition2}
%         \mu+\varepsilon\leq\theta_j,
%     \end{equation}
%     \begin{equation}
%     \label{thetacondition3}
%         KL_fK_0e^{\frac{L_f K_0}{\theta_j-\varepsilon}}\displaystyle\max\left\{\frac{1}{\theta_j-(\alpha-a)-\mu},\frac{1}{\theta_j+(\alpha-a)-\mu}\right\}\leq \frac{1}{2^j},
%     \end{equation}
%     then for any fixed $j\in\mathbb{N}$, there exists an interval $[0,\tilde{t}_{j})$ where $\eta\mapsto G(t,\eta)$ is a $C^{1}$-diffeomorphism for any $t\in[0,\tilde{t}_{j})$.   
% \end{corollary}

\section{Concluding remarks}
\subsection{Context of main results}
A well known approach to address the smoothness of the nonautonomous topological conjugation is the spectral approach,
which emulates the non resonance conditions from the autonomous framework. On the other hand, in this work
we follow an alternative approach which constructs explicitly the topological conjugation via the Green function
and it is important to emphasize that both approaches work with different assumptions regarding the nonlinear perturbation. 

When considering the Green's function approach with nontrivial projectors, namely the expansive/contractive case,
there exists no smoothness results with the exception of the work of N. Jara \cite{Jara}. Our two main results
relax some assumptions of \cite{Jara} and its main features are the following:

The first result of this article, namely Theorem \ref{Teo1}, considers explicitly 
the assumption of nonuniform bounded growth without using the spectrum associated to nonuniform exponential dichotomy. The bounded growth property is essential
to deduce the continuity of the solutions of (\ref{nolin1}) with respect to the initial conditions, as it can be seen in
Lemma \ref{continuidadCIy}. We stress that the previous works on this approach were restricted to the particular case of bounded matrices $A(\cdot)$.

The second main result of this article, namely Theorem \ref{Teo2}, deserves important remarks, which are detailed below:

\medskip

\noindent 1) A careful reading of the proof shows that $\eta \mapsto G(t,\eta)$
verifies \textbf{(TC3)} and $\eta \mapsto \frac{\partial G}{\partial \eta}(t,\eta)$ is continuous in $\mathbb{R}^{n}$ for any fixed $t\in\mathbb{R}_{0}^{+}$. The use of the interval $[0,\tilde{t}_{c})$ is necessary in order to obtain that $\|\Lambda(t,\eta)\|<1$, which is only a sufficient condition ensuring that  $\det \frac{\partial G}{\partial \eta}(t,\eta)\neq 0$.

\medskip

\noindent 2) The conditions (\ref{thetacondition1}), (\ref{thetacondition2}) and (\ref{thetacondition3}) for $\theta$  are consistent with the parameter $b$ and the conditions (\ref{conditions2}) and (\ref{conditions3}), this is mainly due to the fact that parameters $L_f$ and $b$ have a certain degree of freedom respect to the rest of the parameters.  
% In particular, if for example we consider the existence of an increasing sequence $\{b_n\}_{n\in\mathbb{N}}$ in the expression (\ref{reescrituratheta}), then the sequence $\{\theta_n\}_{n\in\mathbb{N}}$ will also be increasing fulfilling all the conditions. In summary, the above result guarantees that if there exist the sequences $\{\theta_n\}_{n\in\mathbb{N}}$ and $\{b_n\}_{n\in\mathbb{N}}$ satisfying the above conditions, then we will get a differentiability result on an interval that will depend on each natural number. Additionally, due to the increasing behavior of these two sequences, the intervals where differentiability will be achieved will also be increasing.

\medskip

\noindent 3) By considering the condition \eqref{constraints-IH}, if $\varepsilon\geq\mu$, then we have that $-\alpha+a+\varepsilon>0$, which implies that $\tilde{t}_{c}=\frac{\ln(c-1)}{-\alpha+a+\varepsilon}>0$. Otherwise, when $\varepsilon<\mu$, the sign of $-\alpha+a+\varepsilon$ remains undetermined. Notice that if $-\alpha+a+\varepsilon\leq 0$ is satisfied, we will obtain that $\|\Lambda(t,\eta)\|<1$ for any $t\geq 0$ and, consequently, $\tilde{t}_{c}=+\infty$. 

\subsection{Comparison with Jara's result}
Whereas in \cite[p.1761]{Jara} the author imposes classical assumptions on the Green's function, namely conditions
\textbf{c2)} and \textbf{c3)} which are necessary to the explicit construction of the topological equivalence, in our work we obtain explicit conditions ensuring these assumptions in (\ref{Green1}) by considering a Banach space with a
weighted norm interlinking the parameters associated to the dichotomy, the  bounded growth properties as well as the nonlinearities.

In \cite[Cor. 4.8]{Jara}, the author obtains a result formally similar to Theorems \ref{Teo1} and \ref{Teo2}
of this article allowing a deeper comparison:

\begin{enumerate}
\item[i)] With respect to the construction of the topological equivalence, the condition \textbf{(c3)} from \cite{Jara} is verified only for nonlinear perturbations which exponentially decrease towards the origin when $t\to +\infty$ while the left estimation of (\ref{Green1}) is satisfied
for unbounded perturbations verifying \textbf{(P2)} with (\ref{conditions1}).
\item[ii)] With respect to the smoothness of the topological equivalence, Theorem \ref{Teo2} guarantees a $C^1$-diffeomorphism $\eta\mapsto G(t,\eta)$ on an interval $[0,\tilde{t}_{c})$, whereas Corollary 4.8 in \cite{Jara} ensures a $C^1$-diffeomorphism on the entire half-line $\mathbb{R}^{+}_{0}$. 
Often, Theorem~2 can be extended further in time  only for a quite restrictive region in the space of parameters.
\end{enumerate}

\bibliographystyle{abbrv}

\bibliography{samplebib}

\end{document}